\documentclass[a4paper,reqno]{amsart}
\usepackage[utf8]{inputenc}
\usepackage{amssymb}
\usepackage{amsmath}
\usepackage{mathtools}
\usepackage{ifpdf}
\usepackage{vmargin}

\ifpdf 
\usepackage[pdftex]{hyperref}
\else
\usepackage[hypertex]{hyperref}
\fi

\hypersetup{
citecolor=green,
menucolor=red,
citebordercolor=0 1 0,
linkbordercolor=1 0 0,
pdfpagemode=UseOutlines,
pdftitle={Minimal mass blow up solutions for a double power nonlinear Schr\"odinger equation},
pdfauthor={Stefan Le Coz <slecoz@math.univ-toulouse.fr>, Yvan Martel <yvan.martel@math.polytechnique.fr>, Pierre Raphaël <praphael@unice.fr>},
pdfsubject={Construction of minimal mass blow up solutions of  Schrödinger equations with double power nonlinearity},
pdfkeywords={Nonlinear Schrödinger equations,blow-up} 
}

\newtheorem{theorem}{Theorem}
\newtheorem{proposition}{Proposition}
\newtheorem{lemma}[proposition]{Lemma}

\theoremstyle{remark}
\newtheorem{remark}{Remark}

\DeclarePairedDelimiter{\norm}{\lVert}{\rVert}

\newcommand{\psld}[2]{\left( #1,#2 \right)_{2}}

\newcommand{\dual}[2]{\left\langle #1,#2 \right\rangle}

\newcommand{\eps}{\varepsilon}

\newcommand{\N}{\mathbb{N}}
\newcommand{\R}{\mathbb{R}}

\DeclareMathOperator{\ModOp}{Mod_{op}}
\DeclareMathOperator{\Mod}{Mod}

\renewcommand{\Re}{\mathrm{Re}}
\renewcommand{\Im}{\mathrm{Im}}
\newcommand{\be}{\begin{equation}}
\newcommand{\ee}{\end{equation}}
\newcommand{\bea}{\begin{eqnarray}}
\newcommand{\eea}{\end{eqnarray}}
\newcommand{\bee}{\begin{eqnarray*}}
\newcommand{\eee}{\end{eqnarray*}}

\def\pa{\partial}
\def\l{\lambda}
\def\fref{\eqref}
\def\lsl{\frac{\lambda_s}{\lambda}}

\begin{document}

\title[Minimal mass blow-up for double power NLS]{Minimal mass blow up solutions for a double power nonlinear Schr\"odinger equation}


\date\today

\author{Stefan Le Coz}
\author{Yvan Martel}
\author{Pierre Raphaël}

\address{Institut de Math\'ematiques de Toulouse,
\newline\indent
Universit\'e Paul Sabatier
\newline\indent
118 route de Narbonne, 31062 Toulouse Cedex 9
\newline\indent
France}
\email{slecoz@math.univ-toulouse.fr}

\address{
Centre de Mathématiques Laurent Schwartz
\newline\indent
\'Ecole Polytechnique
\newline\indent
91 128 Palaiseau cedex
\newline\indent
France}
\email{yvan.martel@polytechnique.edu}

\address{Laboratoire J.A. Dieudonné
\newline\indent
Université de Nice Sophia Antipolis, et Institut Universitaire
de France
\newline\indent
Parc Valrose, 
06108 Nice Cedex 02
\newline\indent
France}
\email{praphael@unice.fr}
  

\begin{abstract}
We consider a nonlinear Schrödinger equation with double power nonlinearity $$i\pa_tu +\Delta u+|u|^{\frac 4d}u+\epsilon |u|^{p-1}u=0, \ \ \epsilon\in\{-1,0,1\}, \ \ 1<p<\frac 4d$$ in $\R^d$ ($d=1,2,3$). Classical variational arguments ensure that $H^1(\R^d)$ data with $\|u_0\|_{2}<\|Q\|_{2}$  lead to global in time solutions, where $Q$ is the ground state of the mass critical problem ($\epsilon=0$).  We are interested by the threshold dynamic $\|u_0\|_{2}=\|Q\|_{2}$ and in particular by the existence of finite time blow up minimal solutions. For $\epsilon=0$, such an object exists thanks to the explicit conformal symmetry,  and is in fact unique from the seminal work \cite{Me93}. For $\epsilon=-1$, simple variational arguments ensure that minimal mass data lead to global in time solutions. We investigate in this paper the case $\epsilon=1$, exhibiting a new class of minimal blow up solutions with blow up rates deeply affected by the double power nonlinearity. The analysis adapts the recent approach \cite{
RaSz11} for 
the construction of minimal blow up elements.
\end{abstract}

\maketitle
 

\section{Introduction}

We consider the following double power nonlinear Schrödinger equation in $\R^d$ 
\begin{equation}
\label{eq:nlseps}
{\rm (NLS)}\ \ \left\{\begin{array}{ll} i\pa_tu+\Delta u+|u|^{\frac{4}{d}}u+\epsilon |u|^{p-1}u=0,\\ u_{|t=0}=u_0,\end{array}\right. \ \ 1<p<1+\frac 4d, \ \ \epsilon\in \{-1,0,1\}.
\end{equation}
This model corresponds to a subcritical perturbation of the classical mass critical problem $\epsilon=0$ which  rules out the scaling symmetry of the problem. It is well-known (see e.g \cite{Ca03} and the references therein) that for any  $u_0\in H^1(\R^d)$, there exists a unique maximal solution $u\in\mathcal{C}((-T_\star,T^\star),H^1(\R^d))$ $\cap$ $\mathcal{C}^1((-T_\star,T^\star),H^{-1}(\R^d))$ of \eqref{eq:nlseps}. Moreover, the mass (i.e. $L^2$ norm) and energy $E$ of the solution are conserved by the flow where:
\begin{equation*}
E(u)=\frac12\norm{\nabla u}_2^2-\frac{1}{2+\frac 4d}\norm{u}_{2+\frac 4d}^{2+\frac 4d}-\epsilon \frac{1}{p+1}\norm{u}_{p+1}^{p+1}.
\end{equation*}
Moreover, there holds the blow up criterion: 
\be
\label{blowoucrit}
 T^\star<+\infty\ \ \mbox{implies}\ \ \lim_{t\uparrow T^\star}\norm{\nabla u(t)}_2=+\infty.
 \ee
In this paper, we are interested in the derivation of a sharp global existence criterion for \eqref{eq:nlseps} in connection with the existence of {\it minimal mass blow up solutions} of \eqref{eq:nlseps}.

\subsection{The mass critical problem} Let us briefly recall the structure of the mass critical problem $\epsilon=0$. In this case, the scaling symmetry $$u_\l(t,x)=\l^{\frac d2}u(\l^2 t, \l x)$$ acts on the set of solutions and leaves the mass invariant $$\|u_\l(t,\cdot)\|_{2}=\|u(\l^2t, \cdot)\|_{2}.$$ From variational argument \cite{We83}, the unique (\cite{BL83,Kw89}) up to symmetry ground state solution to 
\[
-\Delta Q+Q-|Q|^{\frac4d}Q=0, \ \ Q\in H^1(\R^d), \ \ Q>0, \ \ \mbox{$Q$ radial}
\]
attains the best constant in the Gagliardo-Nirenberg inequality $$\|u\|^{2+\frac 4d}_{{2+\frac 4d}}\leq C\|u\|_{2}^{\frac 4d}\|\nabla u\|_{2}^2,$$
so that
\be
\label{gagenergy}
\forall u\in H^1(\R^d), \ \ E_{\rm crit}=\frac12\norm{\nabla u}_2^2-\frac{1}{2+\frac 4d}\norm{u}_{2+\frac 4d}^{2+\frac 4d}\geq \frac 12\|\nabla u\|_{2}^2\left[1-\left(\frac{\|u\|_{2}}{\|Q\|_{2}}\right)^{\frac 4d}\right].
\ee
Together with the conservation of mass and energy and the blow up criterion \eqref{blowoucrit}, this implies the global existence of all solutions with data $\|u_0\|_{2}<\|Q\|_{2}$. In fact, there holds scattering, see \cite{dodson} and references therein.\\
At the threshold $\|u_0\|_{2}=\|Q\|_{2}$, the pseudo-conformal symmetry
\be
\label{pseduocno}
\frac{1}{|t|^{\frac d2}}u\left(\frac 1t,\frac{x}{t}\right)e^{i\frac{|x|^2}{4t}}
\ee applied to the solitary wave solution $u(t,x)=Q(x)e^{it}$  yields the existence of the following explicit minimal  blow up solution
\be\label{st}
S(t,x) = \frac{1}{|t|^{\frac d2}}Q\left(\frac {x}{|t|} \right)e^{-i\frac{|x|^2}{4|t|}}e^{\frac {i}{|t|}},\quad
\|S(t)\|_2= \|Q\|_{2}, \quad \|\nabla S(t)\|_{2}\mathop{\sim}_{t\sim 0^-}\frac 1{|t|}.
\ee
From \cite{Me93}, minimal blow up elements are {\it classified} in $H^1(\R^d)$ in the following sense $$\|u(t)\|_{2}=\|Q\|_{2}\ \ \mbox{and}\ \ T^*<+\infty\ \ \mbox{imply}\ \ u\equiv S$$ up to the symmetries of the flow. Note that the minimal blow up dynamic \eqref{st} can be extended to the super critical mass case $\|u_0\|_{2}>\|Q\|_{2}$ (see  \cite{BW}) and that it corresponds to an unstable threshold dynamics between global in time scattering solutions and finite time blow up solutions in the {\it stable} blow up regime 
\be
\label{logloglaw}
\|\nabla u(t)\|_{2} \mathop{\sim}_{t\sim T^*}\sqrt{\frac{\log|\log|T^*-t||}{T^*-t}}.
\ee
We refer to \cite{MeRaSz13} and references therein for an overview of the existing literature for the $L^2$  critical blow up problem. 

\subsection{\texorpdfstring{The case $\epsilon=-1$}{The case epsilon=-1}} 
Let us now consider the case of a defocusing perturbation. First, there are no solitary waves with subcritical mass $\|u_0\|_{2}<\|Q\|_{2}$ from a standard Pohozaev integration by parts argument. At the threshold, we claim:

\begin{lemma}[Global existence at threshold for $\epsilon=-1$]
\label{lemmathreshodl}
Let $\epsilon=-1$.
Let $u_0\in H^1(\R^d)$ with  $\|u_0\|_{2}=\|Q\|_{2}$, then  the solution of \eqref{eq:nlseps} is global and bounded in $H^1(\R^d)$.
\end{lemma}

The proof follows from standard concentration compactness argument, see Appendix \ref{appendixa}. The global existence criterion of Lemma \ref{lemmathreshodl} is sharp in the sense that for all $\alpha^*>0$, we can build an $H^1(\R^d)$ finite time blow up solution to \eqref{eq:nls} with $\|u_0\|_{2}=\|Q\|_{2}+\alpha^*$ and blow up speed given by the log-log law \eqref{logloglaw}. This is  a consequence of the strong structural stability of the log log regime and the proof would follow the lines of \cite{PlRa07,Ra05,RDuke}.

\subsection{\texorpdfstring{The case $\epsilon=1$}{The case epsilon=1}} We now turn to the case $\epsilon=1$ for the rest of the paper, i.e. we consider the model
\begin{equation}
\label{eq:nls}
 i\pa_tu+\Delta u+|u|^{\frac{4}{d}}u+  |u|^{p-1}u=0 \quad \hbox{where}\quad 1<p<1+\frac 4d.
\end{equation}
 First, from mass and energy conservation,    using \eqref{gagenergy} and \eqref{GNbis}, $H^1(\R^d)$ solutions with $\|u_0\|_{2}<\|Q\|_{2}$ are global and bounded in $H^1(\R^d)$. However,   large time scattering is not true in general, even for small $L^2$ solutions, since there exist arbitrarily small solitary waves.

\begin{lemma}[Small solitary waves]
\label{smallsolitary}
For all $M\in (0,\|Q\|_{2})$, there exists $\omega(M)>0$ and a Schwartz  radially symmetric solution of
 $$\Delta Q_M-\omega(M) Q_M+Q_M^{1+\frac 4d}+Q_M^p=0, \ \ \|Q_M\|_{2}=M.$$
\end{lemma}
The proof follows from classical variational methods, see Appendix \ref{appendixb}.

The main result of this paper is the existence of a minimal mass blow up solution for \eqref{eq:nls}, in contrast with the defocusing case $\epsilon=-1$.

\begin{theorem}[Existence of a minimal blow up element]
\label{thm:1}
Let $d=1,2,3$ and $1<p<1+\frac{4}{d}$. Then for all energy level $E_0\in \R$, there exist $t_0<0$ and a radially symmetric Cauchy data $u(t_0)\in H^1(\R^d)$ with $$\|u(t_0)\|_{2}=\|Q\|_{2}, \ \ E(u(t_0))=E_0,$$ such that the corresponding solution $u(t)$ of \eqref{eq:nls} blows up at time $T^*=0$ with speed:
\be
\label{blowupspeed}
\|\nabla u(t)\|_{2}=\frac{C(p)+o_{t\uparrow 0}(1)}{|t|^{\sigma}}
\ee
for some universal constants $$ \quad \sigma=\frac{4}{4+d(p-1)}\in \left(\tfrac 12 ,1\right), \ \ C(p)>0.$$
\end{theorem}


{\it Comments on the result.}

\medskip


{\it 1. On the existence of minimal elements.}
 Since the pioneering work \cite{Me93}, it has long been believed  that the  existence of a minimal blow up bubble was related to  the exceptional pseudo conformal symmetry \eqref{pseduocno}, or at least to the existence of a sufficiently sharp approximation of it, see \cite{BCD,MMmin}. However, a new methodology to construct minimal mass elements  for a  inhomogeneous (NLS) problem, {\it non perturbative} of critical (NLS), was developed in \cite{RaSz11},  and later successfully applied to problems without any sort of pseudo conformal symmetry, \cite{Boulenger,KLR,MMRII}. More generally, the heart of the matter is to be able to compute 
the trajectory of the solution on the soliton manifold, see  \cite{KMR,MMkdvan} for related problems for two solitary waves motion. The present paper adapts this approach which relies on the {\it direct} computation of the blow up speed and the control of non dispersive bubbles as in \cite{Martelkdv}. 

Observe that  the blow up speed \eqref{blowupspeed}   is quite surprising since it approaches the self simiar blow up speed ${|t|^{-\frac 12}}$ as $p\to \left(1+\frac 4d\right)^-$. \medskip

{\it 2. Uniqueness}. A delicate question investigated in \cite{Boulenger,MMRII,RaSz11} is the uniqueness of the minimal blow up element. Such a uniqueness statement should involve Galilean drifts since the Galilean symmetry applied to \eqref{eq:nls} is an $L^2$ isometry and automatically induces minimal elements with non trivial momentum.  Uniqueness issues lie within the general question of classifying the compact elements of the flow in the Kenig-Merle road map \cite{KM}. A more limited question is to determine the global behavior of the minimal element for negative time, which  is poorly understood in general. Here, at least in the case $E_0\geq 0$, one can see from Virial type estimates that the solution is global in negative time. \medskip

{\it 3. Detailed structure of the singular bubble}. The analysis provides the following detailed structure of the blow up bubble
\be\label{eq:form}
u(t,x)=\frac{1}{\lambda^{\frac{d}{2}}(t)}
Q\left(\frac{x}{\lambda (t)}\right)e^{-i\sigma\frac{|x|^2}{4t}}e^{i\gamma(t)}
+v\left(t,{x}\right)
\ee
where $Q$ is the mass critical ground state, and
$$\lim_{t\rightarrow 0}\norm{v(t)}_{2}=0,\quad
\lambda(t)\sim C_p { |t|^\sigma}\quad \hbox{as $t \to 0^-$},
$$
for some constant $C_p>0$.
Note also that the dimension restriction $d\in \{1,2,3\}$ is for the sake of simplicity but not essential.\medskip

The construction of the minimal blow up element for \eqref{thm:1} can be viewed as  part of a larger program of understanding what kind of blow up speeds are possible for (NLS) type models. Let us repeat that log-log type solutions with super critical mass can be constructed for \eqref{eq:nlseps}, but then the question becomes: do these examples illustrate all  possible blow up types, at least near the ground state profile? The recent series of works \cite{MMRII,MMRIII,MMR1} for the mass critical gKdV equation indicate that this is a delicate problem, and that the role played by the topology used to measure the perturbation is essential. More generally, symmetry breaking perturbations are very common in nonlinear analysis, and while they are expected to be lower order for generic stable blow up dynamics, our analysis shows that they can dramatically influence the structure of unstable threshold dynamics such as in our case  minimal blow up bubbles.\medskip

\noindent{\bf Aknowldedgments.} 
S. Le Coz is partly supported by the  ANR project ESONSE.
Y. Martel and P. Rapha\"el are partly supported by the ERC advanced grant 291214 BLOWDISOL.

\subsection{Notation} Let us collect the main notation used throughout the paper. For the sake of simplicity, we   work in the radial setting only. The $L^2$ scalar product and $L^q$ norm ($q\geq 1$) are  denoted by 
\[
\psld{u}{v}=\Re\left(\int_{\R^d}u(x)\bar v(x) dx\right),\quad \|u\|_{q} = \left(\int_{\R^d} |u|^q\right)^{\frac 1q}.
\]
We fix the notation:
\[
f(z)=|z|^{\frac{4}{d}}z;\quad 
g(z)=|z|^{p-1}z;\quad
F(z)=\frac{1}{\frac{4}{d}+2}|z|^{\frac{4}{d}+2};\quad
G(z)=\frac{1}{p+1}|z|^{p+1}.
\]
Identifying $\mathbb C$ with $\R^2$, we denote  the differential of these functions by $df$, $dg$, $dF$ and $dG$.  Let  $\Lambda$ be the generator of $L^2$-scaling i.e. 
\[
\Lambda=\frac{d}{2}+y\cdot\nabla.
\]
The linearized operator close to $Q$ comes as a matrix\[
L_+:=-\Delta+1-\left(1+\frac{4}{d}\right)Q^{\frac{4}{d}},\qquad
L_-:=-\Delta+1-Q^{\frac{4}{d}}.
\]
and the generalized kernel of $$\begin{pmatrix}0 & L_-\\-L_+&0\end{pmatrix}$$ is non-degenerate and spanned by 
 the symmetries of the problem (see \cite{Kw89,We85} for the original results and \cite{ChGuNaTs07} for a   short proof). It is completely described in $H^1_{\rm rad}(\R^d)$ by  the relations (we define $\rho$ as the unique radial solution to $L_+\rho =|y|^2Q$)
\begin{equation}\label{algebra}
L_-Q=0,\quad L_+\Lambda Q=-2Q,\quad L_-|y|^2Q=-4\Lambda Q,\quad L_+\rho =|y|^2Q.
\end{equation}

Denote by $\mathcal Y$ the set of radially symmetric functions $f\in \mathcal{C}^\infty(\R^d)$ such that
$$
\forall \alpha\in\N^d,\quad \exists C_\alpha,\ \kappa_\alpha>0, \ \forall x\in\R^d,\quad 
|\partial^\alpha f(x)|\leq C_\alpha (1+|x|)^{\kappa_\alpha} Q(x).
$$
It follows from the kernel properties of $L_+$ and $L_-$, and from well-known   properties of the Helmholtz kernel (see \cite{AbSt64} for the properties of Helmholtz kernel (i.e. Bessel and Hankel functions) and  \cite[Appendix A]{CoLe11} or proof of Lemma 3.2 in \cite{MeRaSz14} for related arguments) that
\begin{align}
&\forall g\in \mathcal{Y}, \ \exists f_+\in \mathcal{Y}, \ L_+ f_+ = g,\label{eq:Lp}\\
&\forall g\in \mathcal{Y}, \ \psld{g}{Q}=0, \ \exists f_-\in \mathcal{Y}, \ L_- f_- = g.\label{eq:Lm}
\end{align}
It is also well known (see e.g. \cite{MeRa04,MeRa05,RaSz11,We86}) that $L_+$ and $L_-$ verify the following coercivity property: {\it there exists $\mu>0$ such that  for all $\eps=\eps_1+i\eps_2\in H^1_{\rm rad}(\R^d)$,}  \begin{equation}\label{eq:coercivity}
\dual{L_+\eps_1}{\eps_1}+\dual{L_-\eps_2}{\eps_2}\geq  \mu \norm{\eps}_{H^1}^2-\frac 1\mu\left( \psld{\eps_1}{Q}^2+\psld{\eps_1}{|y|^2Q}^2+\psld{\eps_2}{\rho}^2\right).
\end{equation}
Throughout the paper,  $C$   denotes various positive constants whose exact values may vary from line to line but are of no importance in the analysis. When an inequality is true up to   such a  constant, we also use the notation $\lesssim$, $\gtrsim$ or $\approx$.

\section{Construction of the blow-up profile}

In this section, we define the blow-up profile which is relevant to construct the minimal mass solution -- see Proposition  \ref{prop:profile} below. 

\subsection{Blow up profile}
Let us start with some heuristic arguments justifying the construction.
As usual in blow up contexts, we look for a solution of the following form, with rescaled variables $(s,y)$:
\[
u(t,x)=\frac{1}{\lambda^{\frac{d}{2}}(s)}w(s,y)e^{i\gamma(s)-i\frac{b(s)|y|^2}{4}},\qquad
\frac{ds}{dt}=\frac{1}{\lambda^2},\qquad 
y=\frac{x}{\lambda(s)},
\]
where the function $w$, and the time dependent parameters $\lambda>0$, $b$ and $\gamma$ are to be determined
satisfying the following equation
\begin{multline}\label{eq:nlsscal}
iw_s+\Delta w-w+f(w)+\lambda^{\alpha}g(w)\\-i\left(b+\frac{\lambda_s}{\lambda}\right)\Lambda w
+(1-\gamma_s)w
+(b_s+b^2)\frac{|y|^2}{4}w-b\left(b+\frac{\lambda_s}{\lambda}\right)\frac{|y|^2}{2}w
=0,
\end{multline}
where  
\[
\alpha=2-\frac{d(p-1)}{2}\in (0,2).
\]
Since we look for blow up solutions, the parameter $\lambda(s)$ should converge to zero as $s\to \infty$.
Therefore,
\begin{equation}\label{ww}
w(s,y)=Q(y),\qquad b+\frac{\lambda_s}{\lambda}=b_s+b^2=1-\gamma_s=0
\end{equation}
is a solution of \eqref{eq:nlsscal} at the first order, i.e. when neglecting $\lambda^\alpha |w|^{p-1}w$. 
However, the first order error term $\lambda^\alpha Q^p$ cannot be neglected in the minimal mass blow up analysis (while it could be neglected easily in the log-log regime where $\l\sim e^{-e^{\frac{c}{b}}}$). Therefore, starting from $Q$, we need to look for a refined blow up ansatz. 
Actually, to close the analysis for any $\alpha\in (0,2)$, we need to remove error terms  at any order of $\lambda^\alpha$ and $b$ in the equation of $w$. It is important to note that in the process of constructing the approximate solution,  we cannot exactly solve \eqref{eq:nlsscal} since we need to introduce  new terms in the equation (due to  degrees of freedom necessary to construct the ansatz) that will modify the modulation equations in \eqref{ww}. These terms (gathered in the time dependent function $\theta(s)$ below) are responsible for the specific blow up law obtained in Theorem~\ref{thm:1}. 

\bigskip

Fix $K\in \N$, $K\gg 1$ ($K>20/\alpha$ is sufficient in the proof of Theorem \ref{thm:1}), and
$$
 \Sigma_K= \{(j,k)\in \N^2 \ | \ j+k\leq K \}.
$$

\begin{proposition}\label{prop:profile}
Let $\lambda(s)>0$ and $b(s)\in \R$ be $\mathcal C^1$ functions of $s$ such that $\lambda(s)+|b(s)|\ll 1$.
\\
{\rm (i) Existence of a blow up profile.}
For any $(j,k)\in \Sigma_K$, there exist real-valued functions $P_{j,k}^+\in \mathcal Y$, $P_{j,k}^-\in \mathcal Y$ and $\beta_{j,k}\in\R$ such that
$P(s,y)=\tilde P_K(y;b(s),\lambda(s))$, where $\tilde P_K$ is defined by
\begin{equation}\label{eq:def-P}
   \tilde P_K(y;b,\lambda):=Q(y)+\sum_{(j,k)\in\Sigma_K} b^{2j}\lambda^{(k+1)\alpha} P_{j,k}^+(y)
                +i\sum_{(j,k)\in\Sigma_K}  b^{2j+1}\lambda^{(k+1)\alpha} P_{j,k}^-(y)
\end{equation}
satisfies
\begin{equation*}
i\partial_s P +\Delta P-P+f(P)+\lambda^\alpha g(P)
+\theta \frac{|y|^2}{4}P =\Psi_K
\end{equation*}
where $\theta(s)=\tilde \theta(b(s),\lambda(s))$,
\[
\tilde \theta(b,\lambda) = \sum_{(j,k)\in\Sigma_K} b^{2j}\lambda^{(k+1)\alpha} \beta_{j,k}
\]
and
\begin{align}
\sup_{y\in \R^d} \left(e^{\frac{|y|}{2}}\left( |\Psi_K(y)| + |\nabla\Psi_K(y)|\right)\right) 
 \lesssim \lambda^\alpha\left(\Big|b+\frac{\lambda_s}{\lambda}\Big|+\left|b_s+b^2-\theta\right|\right) 
 +(|b|^2+\lambda^{\alpha})^{K+2} \label{eq:error-term-estimate}.
\end{align}
{\rm (ii) Rescaled blow up profile.} Let
\begin{equation}\label{eq:def-P_b}
P_b(s,y):=P(s,y)e^{-i\frac{b(s)|y|^2}{4}}.
\end{equation}
Then
\begin{multline}
i\partial_s P_b+\Delta P_b-P_b+f(P_b)+\lambda^\alpha g(P_b)
-i\frac{\lambda_s}{\lambda}\Lambda P_b\\=-i\left(\frac{\lambda_s}{\lambda}+b\right)\Lambda P_b+(b_s+b^2-\theta)\frac{|y|^2}{4}P_b+\Psi_K e^{-i\frac{b|y|^2}{4}}.
\label{eq:P_b}\end{multline}
{\rm (iii) Mass and energy properties of the  blow up profile.}
Let
$$
P_{b,\lambda,\gamma}(s,y) = \frac 1{\lambda^{\frac d2}} P_b\left(s,\frac x\lambda\right)e^{i\gamma}.
$$
Then,
\begin{equation}\label{dmass}
  \left|\frac d{ds}\int |P_{b,\lambda,\gamma}|^2\right| \lesssim
  \lambda^\alpha\left(\Big|b+\frac{\lambda_s}{\lambda}\Big|+\left|b_s+b^2-\theta\right|\right) 
  +(|b|^2+\lambda^{\alpha})^{K+2},
\end{equation}
\begin{equation}\label{dener}
  \left|\frac d{ds} {E(P_{b,\lambda,\gamma})}\right|\lesssim
  \frac 1{\lambda^2} \left(\Big|b+\frac{\lambda_s}{\lambda}\Big|+\left|b_s+b^2-\theta\right|
  +(|b|^2+\lambda^{\alpha})^{K+2}\right).
\end{equation}
Moreover, for any $(j,k)\in\Sigma_K,$ there exist  $\eta_{j,k} \in \R$ such that
\begin{equation}\label{eener}
  \left|E(P_{b,\lambda,\gamma}) - \frac{\int |y|^2Q^2}{8} \mathcal{E}(b,\lambda)\right|
  \lesssim \frac{(b^2+\lambda^\alpha)^{K+2}}{\lambda^2},
\end{equation}
where
\be
\label{defee}
  \mathcal{E}(b,\lambda) =  \frac{b^2}{\lambda^2}
  -\frac{2\beta}{2-\alpha} \lambda^{\alpha-2}
  +\lambda^{\alpha-2}\sum_{(j,k)\in \Sigma_K, j+k\geq 1} b^{2j} \lambda^{k\alpha} \eta_{j,k}.
\ee
\end{proposition}

See a similar construction of a blow up profile  at any order of $b$ in \cite{MeRaSz14}.
One sees in \eqref{eq:P_b} the impact of the subcritical nonlinearity $g(u)$ on the blow up law
$b_s+b^2-\theta=0$, which differs from the unperturbed equation $b_s+b^2=0$, and leads to leading order to $\l^\alpha\approx b^2$, see \eqref{eq:dbu}.

\begin{proof} [Proof of Proposition \ref{prop:profile}]
\noindent{\it Proof of {\rm (i)}.} For time dependent functions $\lambda(s)>0$, $b(s)$, we set 
\[
P=Q+\lambda^\alpha Z\quad  \text{where}\quad 
Z=\sum_{(j,k)\in\Sigma_K} b^{2j}\lambda^{k\alpha} P_{j,k}^+
+i\sum_{(j,k)\in\Sigma_K}  b^{2j+1}\lambda^{k\alpha} P_{j,k}^-,
\]
$$\theta(s)=  \sum_{(j,k)\in\Sigma_K} b^{2j}(s)\lambda^{(k+1)\alpha}(s) \beta_{j,k} ,$$
where $P_{j,k}^+\in \mathcal Y$, $P_{j,k}^-\in \mathcal Y$ and $\beta_{j,k}$ are to be determined. 
Set
$$
\Psi_K=i\partial_s P +\Delta P-P+|P|^{\frac4d}P+\lambda^\alpha |P|^{p-1}P
+\theta  \frac{|y|^2}{4}P.
$$
The objective is to choose the unknown functions and parameters so that the error term $\Psi_K$ is controlled as in \eqref{eq:error-term-estimate}.
First, 
\begin{align*}
i P_s& = i \frac{\lambda_s}{\lambda} \sum_{(j,k)\in\Sigma_K} (k+1)\alpha b^{2j} \lambda^{(k+1)\alpha}P_{j,k}^+
+ i b_s \sum_{(j,k)\in\Sigma_K} 2j b^{2j-1}\lambda^{(k+1)\alpha}P_{j,k}^+\nonumber\\
& -  \frac{\lambda_s}{\lambda} \sum_{(j,k)\in\Sigma_K} (k+1)\alpha b^{2j+1}\lambda^{(k+1)\alpha}P_{j,k}^-
- b_s \sum_{(j,k)\in\Sigma_K} (2j+1) b^{2j}\lambda^{(k+1)\alpha}P_{j,k}^-\nonumber\\
&=-i \sum_{(j,k)\in\Sigma_K} (k+1)\alpha b^{2j+1}\lambda^{(k+1)\alpha}P_{j,k}^+\nonumber\\
&- i \left( b^2-\sum_{(j',k')\in\Sigma_K} b^{2j'}\lambda^{(k'+1)\alpha} \beta_{j',k'}\right)   \sum_{(j,k)\in\Sigma_K} 2j b^{2j-1}\lambda^{(k+1)\alpha}P_{j,k}^+\nonumber\\
& +   \sum_{(j,k)\in\Sigma_K} (k+1)\alpha b^{2(j+1)}\lambda^{(k+1)\alpha}P_{j,k}^-\nonumber\\
&+\left( b^2-\sum_{(j',k')\in\Sigma_K} b^{2j'}\lambda^{(k'+1)\alpha} \beta_{j',k'}\right)  \sum_{(j,k)\in\Sigma_K} (2j+1) b^{2j}\lambda^{(k+1)\alpha}P_{j,k}^-+\Psi^{P_s}
\end{align*}
where
\begin{align}
\Psi^{P_s}
&=\left( \frac{\lambda_s}{\lambda}+b\right) \sum_{(j,k)\in\Sigma_K} (k+1)\alpha b^{2j}\lambda^{(k+1)\alpha}\left(iP_{j,k}^+ - b P_{j,k}^-\right)\nonumber \\
&+ \left(b_s+b^2-\theta \right) \sum_{(j,k)\in\Sigma_K}  b^{2j-1}\lambda^{(k+1)\alpha} \left(2j i P_{j,k}^+-(2j+1)b P_{j,k}^-\right).
\label{Ps}
\end{align}
We rewrite
\begin{align*}
i P_s& =-i\sum_{(j,k)\in\Sigma_K} \left((k+1)\alpha+2j\right) b^{2j+1}\lambda^{(k+1)\alpha}P_{j,k}^+
\\& + i \sum_{j,k\geq 0} b^{2j+1}\lambda^{(k+1)\alpha}F_{j,k}^{P_s,-}
+ \sum_{j,k\geq 0} b^{2j}\lambda^{(k+1)\alpha}F_{j,k}^{P_s,+}
+\Psi^{P_s},
\end{align*}
where for $j,k\geq 0$,
$F_{j,k}^{P_s,\pm}$ depends on various functions $P_{j',k'}^\pm$ and parameters $\beta_{j',k'}$ for 
$(j',k')\in \Sigma_K$ such that either $k'\leq k-1$ and $j'\leq j+1$ or  $k'\leq k$ and $j'\leq j-1$.
Only a finite number of these functions are nonzero.

\bigskip

Next, using $\Delta Q - Q + Q^{\frac 4d+1}=0$, we get
\begin{multline*}
\Delta P - P + |P|^{\frac 4d} P  =
- \sum_{(j,k)\in\Sigma_K} b^{2j}  \lambda^{(k+1)\alpha} L_+ P_{j,k}^+
-i\sum_{(j,k)\in\Sigma_K} b^{2j+1}\lambda^{(k+1)\alpha} L_- P_{j,k}^-\\
+f(Q+\lambda^\alpha Z)-f(Q)-\lambda^\alpha df(Q)Z.
\end{multline*}
Let
\begin{multline}
	\Psi^f= f(Q+\lambda^\alpha Z)-\sum_{k=0}^K d^kf(Q)(\lambda^\alpha Z,\dots,\lambda^\alpha Z)\\
	=
	|Q+\lambda^\alpha Z|^{\frac 4d}(Q+\lambda^\alpha Z) 
	- Q^{\frac 4d+1} \left(1+\sum_{n=1}^{K+1} \frac{\left(\frac 2d+1\right)\left(\frac 2d\right)\ldots\left(\frac 2d-n+2\right)}{n!} \left(\frac{\lambda^\alpha Z}{Q}\right)^n \right) 
\\   \times \left(1+\sum_{n'=1}^{K+1} \frac{\left(\frac 2d\right)\left(\frac 2d-1\right)\ldots\left(\frac 2d-n'+1\right)}{n'!} \left(\frac{\lambda^\alpha \overline Z}{Q}\right)^{n'}\right).
\label{psif}
\end{multline}
Then 
\begin{align*}
\Delta P - P + |P|^{\frac 4d} P & =
- \sum_{(j,k)\in\Sigma_K} b^{2j}  \lambda^{(k+1)\alpha} L_+ P_{j,k}^+
-i\sum_{(j,k)\in\Sigma_K} b^{2j+1}\lambda^{(k+1)\alpha} L_- P_{j,k}^-\\
& + i \sum_{j\geq 0,k\geq 1} b^{2j+1}\lambda^{(k+1)\alpha}F_{j,k}^{f,-}
+ \sum_{j,k\geq 0} b^{2j}\lambda^{(k+1)\alpha}F_{j,k}^{f,+}
+\Psi^f.
\end{align*}
where  for $j,k\geq 0$, 
$F_{j,k}^{f,\pm}$ depends on $Q$ and on various functions $P_{j',k'}^\pm$ for $(j',k')\in \Sigma_K$ such that $k'\leq k-1$ and $j'\leq j$.

\bigskip

Using a similar argument for $\lambda^{\alpha}|P|^{p-1} P$, we obtain
\begin{align*}
\lambda^{\alpha}|P|^{p-1} P & =
 i \sum_{j\geq 0,k\geq 1} b^{2j+1}\lambda^{(k+1)\alpha}F_{j,k}^{g,-}
+ \sum_{j\geq 0,k\geq 1} b^{2j}\lambda^{(k+1)\alpha}F_{j,k}^{g,+}
+\Psi^g,
\end{align*}
where 
\begin{multline*}
\Psi^g = \lambda^\alpha\left( |Q+\lambda^\alpha Z|^{p-1}(Q+\lambda^\alpha Z) 
	- Q^{\frac p2+1} \left(1+\sum_{n=1}^{K+1} \frac{\left(\frac p2\right)\left(\frac p2-1\right)\ldots\left(\frac p2-n+1\right)}{n!} \left(\frac{\lambda^\alpha Z}{Q}\right)^n \right) 
\right.\\ \left.  \times \left(1+\sum_{n=1}^{K+1} \frac{\left(\frac p2-1\right)\left(\frac p2-2\right)\ldots\left(\frac p2-n\right)}{n!} \left(\frac{\lambda^\alpha \overline Z}{Q}\right)^n\right)
\right),
\end{multline*}
and 
where  for $j,k\geq 0$,
$F_{j,k}^{g,\pm}$ depends on $Q$ and on various functions $P_{j',k'}^\pm$ for $(j',k')\in \Sigma_K$ such that $k'\leq k-1$ and $j'\leq j$.

\bigskip

Finally,
\begin{align*}
\theta \frac {|y|^2}4 P & = \left( \sum_{(j,k)\in\Sigma_K} b^{2j}  \lambda^{(k+1)\alpha} \beta_{j,k} \right) \frac {|y|^2}4 Q \\&+ i \sum_{j,k\geq 0} b^{2j+1}\lambda^{(k+1)\alpha}F_{j,k}^{\theta,-} + \sum_{j,k\geq 0} b^{2j}\lambda^{(k+1)\alpha}F_{j,k}^{\theta,+},
\end{align*}
where 
$F_{j,k}^{\theta,\pm}$ depends on $Q$ and on various functions $P_{j',k'}^\pm$ or parameters $\beta_{j',k'}$ for $(j',k')\in \Sigma_K$ such that $k'\leq k-1$ and $j'\leq j$.\\
Combining these computations, we obtain
\begin{align*}
\Psi_K & = 
- \sum_{(j,k)\in\Sigma_K} b^{2j}  \lambda^{(k+1)\alpha} \left( L_+ P_{j,k}^+ - F_{j,k}^{+} -\beta_{j,k} |y|^2 Q \right) \\
& - i\sum_{(j,k)\in\Sigma_K} b^{2j+1}\lambda^{(k+1)\alpha} \left( L_- P_{j,k}^- - F_{j,k}^{-} +((k+1)\alpha +2j) P_{j,k}^+ \right) \\
& +\Psi^{>K} +\Psi^{P_s}+\Psi^f+\Psi^g,
\end{align*}
where
$$
F_{j,k}^{\pm}=F_{j,k}^{P_s,\pm}+F_{j,k}^{f,\pm}+F_{j,k}^{g,\pm}+F_{j,k}^{\theta,\pm},
$$
and
$$
\Psi^{>K}= \sum_{j,k>0, \ (j,k)\not\in\Sigma_K} b^{2j}  \lambda^{(k+1)\alpha}  F_{j,k}^{+} 
+ i\sum_{j,k>0, \ (j,k)\not\in\Sigma_K} b^{2j+1}\lambda^{(k+1)\alpha}  F_{j,k}^{-}.
$$
(Note that the series in the expression of $\Psi^{>K}$ contains only a finite number of terms.)
Now, for any $(j,k)\in \Sigma_K$, we want to choose recursively $P_{j,k}^\pm \in \mathcal Y$  and $\beta_{j,k}$ to solve the system
\[ (S_{j,k})\qquad \left\{
\begin{array}{l}
 L_+ P_{j,k}^+ - F_{j,k}^{+} -\beta_{j,k} |y|^2 Q =0\\
  L_- P_{j,k}^- - F_{j,k}^{-} +((k+1)\alpha +2j) P_{j,k}^+ =0,
\end{array}\right.
\]
where $F_{j,k}^\pm$ are source terms depending of previously determined $P_{j',k'}^\pm$
and $\beta_{j',k'}$.
We argue by a suitable induction argument on the two parameters $j$ and $k$.
For $(j,k)=(0,0)$, we see that the system writes
\begin{align*}
& L_+ P_{0,0}^+ - Q^p -\beta_{0,0} |y|^2 Q =0\\
&  L_- P_{0,0}^-   + \alpha  P_{0,0}^+=0 ,
\end{align*}
(the term $Q^p$ in the first line is coming from $\Psi^g$).
By \eqref{eq:Lm}, for any $\beta_{0,0}\in\R$, there exists a unique $P_{0,0}^+\in \mathcal Y$ so that $L_+ P_{0,0}^+ - Q^p -\beta_{0,0} |y|^2 Q =0$.
We choose $\beta_{0,0}\in\R$ so that  
\[
\psld{P_{0,0}^+}{Q}=-\frac 12\psld{L_+ P_{0,0}^+}{\Lambda Q} =-\frac12 \psld{Q^p+\beta_{0,0}\frac{|y|^2}{4}Q}{ \Lambda Q}=0
\]
(recall from \eqref{algebra} that $L _+\Lambda Q = -2 Q$), which gives
\begin{equation}\label{def-beta}
\beta:=\beta_{0,0}=-\frac{4\psld{Q^p}{\Lambda Q}}{\psld{|y|^2Q}{\Lambda Q}}=\frac{2d(p-1)}{p+1}\frac{\norm{Q}_{p+1}^{p+1}}{\norm{yQ}_2^2}>0.
\end{equation}
By \eqref{eq:Lm}, there exists $P_{0,0}^-\in \mathcal Y$ (unique up to the addition of $cQ$) such that $L_- P_{0,0}^-   + \alpha  P_{0,0}^+=0$.
Now, we assume that for some $(j_0,k_0)\in\Sigma_K$, the following assertion is true:

\bigskip

\noindent $H(j_0,k_0)$ : {\sl for all $(j,k)\in\Sigma_K$ such that either
$k<k_0$, or $k=k_0$ and $j<j_0$, the system $(S_{j,k})$ has a solution $(P_{j,k}^+,P_{j,k}^-,\beta_{j,k})$,  $P_{j,k}^\pm \in \mathcal Y$.} 

\bigskip

In view of the definition of $F_{j_0,k_0}^\pm$, $H(j_0,k_0)$ implies in particular that $F_{j_0,k_0}^\pm\in\mathcal Y$. We now solve the system $(S_{j_0,k_0})$ as before. By \eqref{eq:Lm}, for any $\beta_{j_0,k_0}\in\R$, there exists a unique $P_{j_0,k_0}^+\in \mathcal Y$ so that $L_+ P_{j_0,k_0}^+ - F_{j_0,k_0}^{+} -\beta_{j_0,k_0} |y|^2 Q =0$.
We uniquely choose $\beta_{j_0,k_0}\in\R$ so that  
\[
\psld{ - F_{j_0,k_0}^{-} +((k_0+1)\alpha +2j_0) P_{j_0,k_0}^+}{Q}=0.
\]
By \eqref{eq:Lm}, there exists $P_{j_0,k_0}^-\in \mathcal Y$ (unique up to the addition of $cQ$) such that $L_- P_{j_0,k_0}^-   - F_{j_0,k_0}^{-} +((k_0+1)\alpha +2j_0) P_{j_0,k_0}^+=0$.
In particular, we have proved that
if $j_0<K$, then $H(j_0,k_0)$ implies $H(j_0+1,k_0)$, and $H(K,k_0)$ implies $H(1,k_0+1)$.
This is enough to complete an induction argument on the two parameters $(j,k)$.
Therefore, system $(S_{j,k})$ is solved for all $(j,k)\in \Sigma_K$.\\
It remains to estimate $\Psi_K$ and $\nabla \Psi_K$.
It is straightforward to check that
\begin{align*}
& \sup_{y\in \R^d} \left(e^{\frac{|y|}{2}}\left( |\Psi^{P_s}(y)|+  |\nabla\Psi^{P_s}(y)|\right)\right) \lesssim  \lambda^\alpha \left(\left| \frac{\lambda_s}{\lambda}+b\right|+
\left|b_s+b^2-\theta\right|\right).
\end{align*}
Next, we claim
\begin{equation}\label{eq:taylor}
  |\Psi^f|\lesssim \left(\lambda^{(K+2)\alpha}+\lambda^\alpha b^{2K+2}\right) Q.
\end{equation}
Indeed, first, if $y$ is such that $\left|\lambda^\alpha\frac{Z(y)}{Q(y)}\right|<\frac12$ then the result follows from \eqref{psif} and a  order Taylor expansion of order $K+1$ of 
$(1+\frac{\lambda^\alpha Z}{Q})^{\frac 2d+1}$ and  $(1+\frac{\lambda^\alpha Z}{Q})^{\frac 2d}$.
Second, if on the contrary,  $\left|\lambda^\alpha\frac{Z(y)}{Q(y)}\right|\geq \frac12$, then, since
$Z\in \mathcal Y$, we have, for such $y$,
\[
 Q(y)\leq 2\lambda^\alpha| Z(y)| \lesssim \lambda^\alpha (1+|y|^\kappa) Q(y)
\quad \hbox{and so}\quad
Q(y)+|Z(y)| \lesssim e^{-\frac 12 \lambda^{\alpha/\kappa}},
\]
which completes the proof of \eqref{eq:taylor}.
The proofs of estimates  for $\nabla \Psi_f$, $\Psi_g$ and $\nabla \Psi_g$ are similar.
Finally the following estimates for $\Psi^{>K}$ and $\nabla\Psi^{>K}$ are clear:
$$|\Psi^{>K}|+|\nabla \Psi^{>K}| \lesssim \left(\lambda^{(K+2)\alpha} + \lambda^\alpha |b|^{2K+2}\right)Q^{\frac 12}.
$$
The result follows from  $K\geq \frac {20}\alpha$.\medskip

\noindent{\it Proof of {\rm (ii)}.} This is a straightforward computation which is left to the reader.\\
\noindent{\it Proof of {\rm (iii)}.} To prove \eqref{dmass}, we hit \eqref{eq:P_b} with $iP_b$ and compute using the critical relation $(P,\Lambda P)_2=0$:
$$\frac12 \frac{d}{ds}\|P_b\|_{2}^2=(i\pa_sP_b,iP_b)_2=(\Psi_Ke^{-i\frac{b|y|^2}{4}},iP_b)$$ and \eqref{dmass} follows from \eqref{eq:error-term-estimate}. For  \eqref{dener}, we have from scaling:
\[
 E(P_{b,\lambda,\gamma})=\frac{1}{\lambda^2}\left(\frac 12\int|\nabla P_b|^2-\int F(P_b)-\lambda^\alpha\int G(P_b)\right)=:\frac{1}{\lambda^2}\tilde E(\l, P_b)
\]
Therefore, 
\begin{equation}\label{eq:calc-ener1}
\frac{d}{ds} E(P_{b,\lambda,\gamma})=\frac{1}{\lambda^2}\left(-2\frac{\lambda_s}{\lambda}\tilde E(\l,P_b)+\dual{\tilde E'(\l,P_b)}{\partial_sP_b}-\alpha \l^\alpha\lsl\int G(P_b) \right).
\end{equation}
Using the equation \eqref{eq:P_b} of $P_b$, we compute:
\begin{multline}\label{eq:calc-ener2}
\dual{\tilde E'(\l,P_b)}{\partial_sP_b}
=\frac{\lambda_s}{\lambda}\dual{\tilde E'(\l,P_b)}{\Lambda P_b}-\left(\frac{\lambda_s}{\lambda}+b\right)\dual{\tilde E'(\l,P_b)}{\Lambda P_b}
\\+(b_s+b^2-\theta)\dual{i\tilde E'(\l,P_b)}{\frac{|y|^2}{4}P_b}
+\dual{i\tilde E'(\l,P_b)}{\Psi_Ke^{-i\frac{b|y|^2}{4}}}.
\end{multline}
We now integrate by parts to estimate
\begin{equation}\label{eq:calc-ener3}
\dual{\tilde E'(\l,P_b)}{\Lambda P_b}=\int |\nabla P_b|^2-2\int F(P_b)-\frac{d(p-1)}{2}\int G(P_b)=2\tilde E(\lambda,P_b)+\alpha\lambda^\alpha\int G(P_b),
\end{equation}
where we have used $\alpha=2-\frac{d(p-1)}{2}$, from which:
\bee
\frac{d}{ds} E(P_{b,\lambda,\gamma})&= & \frac{1}{\l^2}\left[-2\frac{\lambda_s}{\lambda}\tilde E(\lambda,P_b)-\alpha\lambda^\alpha\lsl\int G(P_b)+\lsl\left[2\tilde E(\lambda,P_b)+\alpha\lambda^\alpha\int G(P_b)\right]\right]\\
& + & \frac{1}{\l^2}O\left(\left|\lsl+b\right|+|b_s+b^2-\theta|+(b^2+\lambda^\alpha)^{K+2}\right).
\eee
The estimate \eqref{dener} on the time-derivative of the energy then follows from \eqref{eq:calc-ener1}, \eqref{eq:calc-ener2}, \eqref{eq:calc-ener3}, and \eqref{eq:error-term-estimate}. \\
Next,
\begin{align*}
\lambda^2 E(P_{b,\lambda,\gamma})&=\frac 12\int|\nabla P_b|^2-\int F(P_b)-\lambda^\alpha\int G(P_b)\\
&=\frac 12\int |\nabla P|^2+\frac{b^2}{8}\int|y|^2|P|^2-\int F(P)-\lambda^\alpha\int G(P).
\end{align*}
Thus, replacing $P=Q+\lambda^\alpha Z$,
\begin{align*}
\lambda^2 E(P_{b,\lambda,\gamma})
&=\frac 12\int |\nabla Q|^2-\int F(Q)+\frac{b^2}{8}\int|y|^2Q^2-\lambda^\alpha\int G(Q)\\
& +\lambda^\alpha \int (-\Delta Q-f(Q)){\rm Re}Z-\lambda^{2\alpha}\int g(Q){\rm Re}Z+\frac{b^2}{4}\lambda^\alpha\int|y|^2Q{\rm Re}Z\\
&+\frac{\lambda^{2\alpha}}2\int|\nabla Z|^2+\frac{b^2\lambda^{2\alpha}}{8}\int|y|^2|Z|^2
-\int\left\{F(Q+\lambda^\alpha Z)-F(Q)-\lambda^\alpha f(Q){\rm Re}Z\right\}\\
&-\lambda^\alpha\int\left\{G(Q+\lambda^\alpha Z)-G(Q)-\lambda^\alpha g(Q){\rm Re}Z\right\}.
\end{align*}
On the one hand, we recall that from Pohozaev identity,
\[
\frac 12\int |\nabla Q|^2-\int F(Q)=0,
\]
and from the definition \eqref{def-beta} of $\beta_{0,0}$, 
\[
\int G(Q)=\frac{\beta}{2d(p-1)}\int |y|^2Q^2=\frac {\beta}{4(2-\alpha)}\int|y|^2Q^2
\]
 and moreover
\[
\Delta Q+f(Q)=Q.
\]
On the other hand, we observe, since $\int P_{0,0}^+ Q=0$,
$$\lambda^{\alpha} \int ZQ = \lambda^\alpha \sum_{(j,k)\in \Sigma_K,j+k\geq 1} b^{2j}\lambda^{k\alpha}\eta_{j,k}^{\rm I},$$
for some $\eta_{j,k}^{\rm I}\in \R$;
$$\lambda^{2\alpha}\int Zg(Q)=\lambda^{\alpha}\sum_{(j,k)\in \Sigma_K,k\geq 1}b^{2j}\lambda^{k\alpha}\eta_{j,k}^{\rm II},$$
for some $\eta_{j,k}^{\rm II}\in \R$;
$$\lambda^{\alpha}b^2\int |y|^2Q {\rm Re}Z=\lambda^{\alpha}\sum_{(j,k)\in \Sigma_K, j\geq 1}b^{2j}\lambda^{k\alpha}\eta_{j,k}^{\rm III},$$
for some $\eta_{j,k}^{\rm III}\in \R$;
$$\lambda^{2\alpha}\int |\nabla Z|^2+\frac{b^2\lambda^{2\alpha}}{8}\int |y|^2Z^2
=\lambda^{\alpha}\sum_{(j,k)\in \Sigma_K,j\geq 1,k\geq 0}b^{2j}\lambda^{k\alpha}\eta_{j,k}^{\rm IV},$$
for some $\eta_{j,k}^{\rm IV}\in \R$.
Moreover, by Taylor expansion as before, for some $\eta_{j,k}^{\rm V},\eta_{j,k}^{\rm IV}\in \R$
\begin{align*}
 \left|\int\left\{F(Q+\lambda^\alpha Z)-F(Q)-\lambda^\alpha f(Q){\rm Re}Z 
-\lambda^{\alpha}\sum_{(j,k)\in \Sigma_K, k\geq 1}b^{2j}\lambda^{k\alpha}\eta_{j,k}^{\rm V} \right\}\right| 
 \lesssim \lambda^{(K+2)\alpha},\end{align*}
\begin{align*}
  \left|\lambda^\alpha\int\left\{G(Q+\lambda^\alpha Z)-G(Q)-\lambda^\alpha g(Q){\rm Re}Z 
-\lambda^{\alpha}\sum_{(j,k)\in \Sigma_K,k\geq 2}b^{2j}\lambda^{k\alpha}\eta_{j,k}^{\rm VI} \right\}\right| 
  \lesssim \lambda^{(K+2)\alpha}.\end{align*}
Gathering these computations, we obtain \eqref{eener}.
\end{proof}

\subsection{Approximate blow up law}
For simplicity of notation, we set 
$$\beta=\beta_{0,0}= \frac{2d(p-1)}{p+1}\frac{\norm{Q}_{p+1}^{p+1}}{\norm{yQ}_2^2}.$$
First, we  find a relevant solution to   the following approximate  system
\begin{equation}\label{eq:sys}
b_s+b^2-\beta\lambda^{\alpha}=0,\quad
b+\frac{\lambda_s}{\lambda}=0.
\end{equation}
Indeed, for $|b|+\lambda \ll 1$,  $\beta\lambda^\alpha$ is the main term in $\theta$, and the only term in $\theta$ that will modify at the main   order the blow up rate.

\begin{lemma}
Let
\begin{equation}\label{eq:dbu}
\lambda_{\rm app}(s)=\left(\frac\alpha2\sqrt{\frac{2\beta}{2-\alpha}}\right)^{-\frac2\alpha} s^{-\frac2\alpha},
\quad b_{\rm app}(s)={\frac2{\alpha s}}.
\end{equation}
Then $(\lambda_{\rm app}(s),b_{\rm app}(s))$ solves \eqref{eq:sys} for $s>0$.
\end{lemma}

\begin{proof}
We compute:
$$
\left(\frac{b^2}{\lambda^2}\right)_s=2\frac{b}{\lambda}\frac{b_s+b^2}{\lambda}
=-2\beta\frac{\lambda_s}{\lambda}\lambda^{\alpha-2}
,$$ 
and so
\begin{equation}\label{choice}
\frac{b^2}{\lambda^2}-\frac{2\beta}{2-\alpha}\lambda^{\alpha-2}=c_0.
\end{equation}
Taking the constant $c_0=0$, and using $b=-\frac{\lambda_s}{\lambda}>0$, we find
$$
\frac{\lambda_s}{\lambda^{1+\frac{\alpha}{2}}}=\sqrt{\frac{2\beta}{2-\alpha}}.
$$
Therefore,
\[
 \lambda(s)=\left(\frac\alpha2\sqrt{\frac{2\beta}{2-\alpha}}\right)^{-\frac2\alpha} s^{-\frac2\alpha},
\quad b(s)=-\frac{\lambda_s}{\lambda}(s)={\frac2\alpha}\frac 1s\]
is solution of \eqref{eq:sys}.
\end{proof}

\begin{remark}\label{rk:9}
We now express this solution in the  time variable $t_{\rm app}$ related to $\lambda_{\rm app}$. Let
\[
 {dt_{\rm app}}= {\lambda_{\rm app}^2} ds=\left(\frac\alpha2\sqrt{\frac{2\beta}{2-\alpha}}\right)^{-\frac4\alpha} s^{-\frac4\alpha} ds. 
\]
Therefore (with the convention that $t_{\rm app}\to 0^-$ as $s\to +\infty$)
\be
\label{deftapp}
t_{\rm app} =- C_s s^{-\frac {4-\alpha}{\alpha}} \quad \hbox{where} \quad C_s =  \frac \alpha{4-\alpha}\left(\frac\alpha2\sqrt{\frac{2\beta}{2-\alpha}}\right)^{-\frac4\alpha}.
\ee
As a consequence, we obtain for $t_{\rm app}<0$, 
\begin{equation}\label{Cbumille}
\lambda_{\rm app}(t_{\rm app})=C_\lambda |t_{\rm app}|^{\frac{2}{4-\alpha}} \quad \hbox{where} \quad 
C_\lambda  =\left(\frac {4-\alpha}{\alpha} C_s^{-\frac {\alpha}{4-\alpha}}\right)^{\frac 12},
\end{equation} 
\begin{equation}\label{Cbu2}
b_{\rm app}(t_{\rm app})=C_b|t_{\rm app}|^{\frac{\alpha}{4-\alpha}}, \quad
\hbox{where} \quad C_b = \frac 2{\alpha} C_s^{-\frac \alpha{4-\alpha}}.
\end{equation}
\end{remark}

Now, we choose suitable initial conditions $b_1$ and $\lambda_1$ for $b(s)$ and $\lambda(s)$ at some large time $s_1$, first to adjust the value of the energy of $P_{b,\lambda,\gamma}$ (up to  the small error term in \eqref{eener}) and second to be able to close the perturbed dynamical system of $(\lambda,b)$ at the end of the proof (see proof of Lemma \ref{lem:bootstrap} below).
Let $E_0\in \R$ and
\begin{equation*}
C_0=\frac{8E_0}{\int |y|^2Q^2}.
\end{equation*}
Fix $0<\lambda_0\ll 1$  such that $\frac {2\beta}{2-\alpha} + C_0 \lambda_0^{2-\alpha} >0$.
For $\lambda\in (0,\lambda_0]$, let
\begin{equation}\label{defmF}
\mathcal F(\lambda)=\int_\lambda^{\lambda_0}\frac{d\mu}{\mu^{\frac{\alpha}{2}+1}\sqrt{\frac{2\beta}{2-\alpha}+C_0\mu^{2-\alpha}}}.
\end{equation}
Note that the function $\mathcal F$ is related to the resolution of the system \eqref{choice}
for $c_0=C_0$, see proof of Lemma \ref{lem:bootstrap}.

\begin{lemma}\label{le:bu}
Let $s_1\gg1$. There exist $b_1$ and $\lambda_1$ such that
\begin{align}
\label{app_ini}
& \left|\frac {\lambda_1^{\frac \alpha 2}}{\lambda_{\rm app}^{\frac \alpha 2}(s_1)} - 1 \right|
+ \left|\frac{b_1}{b_{\rm app}(s_1)}-1 \right|\lesssim s_1^{-\frac 12}+ s_1^{2-\frac 4\alpha},
\\& 
\mathcal{F}(\lambda_1)=s_1,\quad 
\mathcal{E}(b_1,\lambda_1)=C_0.\label{app_ini2}
\end{align}
\end{lemma}

\begin{proof}
First, we choose $\lambda_1$. Note that $\mathcal F$ is a decreasing function of $\lambda$ satisfying
$\mathcal F(\lambda_0)=0$ and $\lim_{\lambda \downarrow 0} \mathcal F(\lambda)=+\infty$.
Thus, there exists a unique $\lambda_1\in (0,\lambda_0)$ such that
$\mathcal F(\lambda_1)=s_1$.

For $\lambda\in (0,\lambda_0]$,
\begin{align*}
\left|\mathcal F(\lambda)-\frac{2}{\alpha\sqrt{\frac{2\beta}{2-\alpha}} \lambda^{\frac{\alpha}{2}}} \right|&\lesssim 1+\left|\int_\lambda^{\lambda_0}\frac{d\mu}{\mu^{\frac{\alpha}{2}+1}}\left[\frac{1}{\sqrt{\frac{2\beta}{2-\alpha}+C_0\mu^{2-\alpha}}}-\frac{1}{\sqrt{\frac{2\beta}{2-\alpha}}}\right]\right|\nonumber\\
&\lesssim 1+\int_\lambda^{\lambda_0}\frac{d\mu}{\mu^{1+\frac{\alpha}{2}-(2-\alpha)}}.
\end{align*}
Thus,
\begin{equation*}
\left|\mathcal F(\lambda)-\frac{2}{\alpha\sqrt{\frac{2\beta}{2-\alpha}} \lambda^{\frac{\alpha}{2}}} \right| 
\lesssim \left\{\begin{array}{ll}
 1  & \hbox{for $\alpha\in (0,\frac 43)$,}\\
 |\log \lambda| &\hbox{for $\alpha=\frac 43$,}\\
 \lambda^{2-\frac {3\alpha}2}&\hbox{for $\alpha\in (\frac 43,2)$.}
 \end{array}\right.\end{equation*}
To simplify, we will use the non sharp but sufficient estimate
\begin{equation}\label{using}
	\left|\mathcal F(\lambda)-\frac{2}{\alpha\sqrt{\frac{2\beta}{2-\alpha}} \lambda^{\frac{\alpha}{2}}} \right| \lesssim \lambda^{-\frac \alpha 4}+ \lambda^{2-\frac {3\alpha}2}.
\end{equation}
Applied to $\lambda_1$, it gives
$$
\left|s_1 - \frac{2}{\alpha\sqrt{\frac{2\beta}{2-\alpha}} \lambda_1^{\frac{\alpha}{2}}} \right|
\lesssim \lambda_1^{-\frac \alpha 4}+ \lambda_1^{2-\frac {3\alpha}2}\quad \hbox{and thus}\quad
\left|\frac {\lambda_1^{\frac \alpha 2}}{\lambda_{\rm app}^{\frac \alpha 2}(s_1)} - 1 \right|
\lesssim s_1^{-\frac12}+ s_1^{2-\frac 4\alpha}.
$$

Second, we choose $b_1$.
From the definition of $\mathcal{E}$, we have
\begin{align*}
h(b):=\lambda_1^2 \mathcal{E}(b,\lambda_1)
& =b^2-\left(\frac{2}{\alpha s_1}\right)^2 - \frac {2\beta}{2-\alpha} \left(\lambda_1^\alpha
-\lambda_{\rm app}^\alpha(s_1)\right)
+\lambda_1^\alpha \sum_{(j,k)\in\Sigma_K,\  j+k\geq 1} b^{2j}\lambda_1^{-k\alpha}  \eta_{j,k}\\
& =b^2-\left(\frac{2}{\alpha s_1}\right)^2  + O(s_1^{-\frac 52}) +O(s_1^{-\frac 4\alpha}).
\end{align*}
Observe that
$$
|h(b_{\rm app}(s_1))|\lesssim  s_1^{-\frac 4\alpha},\quad |h'(b_{\rm app}(s_1))| \geq 2b_{\rm app}(s_1)+O(s_1^{-3})
\geq {s_1^{-1}}.
$$
Since $\lambda_1^2\approx s_1^{-\frac 4\alpha}$, it follows that
 there exists a unique $b_1$ such that
$$
|b_1-b_{\rm app}(s_1)|\lesssim    s_1^{-\frac 32}+ s_1^{1-\frac 4\alpha},
\quad h(b_1)=C_0 \lambda_1^2,
$$
and so
$
\mathcal{E}(b_1,\lambda_1)= C_0 .
$
\end{proof}

\section{Existence proof assuming uniform  estimates}

This section is devoted to the proof of Theorem \ref{thm:1} by a compactness argument, assuming uniform  estimates on specific solutions of \eqref{eq:nls}. These estimates are given  in   Proposition \ref{prop:uniform-estimates-rescaled-time}.  

\subsection{Uniform estimates in rescaled time variable}

The rescaled time depending on a suitable modulation of the solution $u(t)$, we first recall without proof the following standard result (see e.g. \cite{MeRa05}).

\begin{lemma}[Modulation]\label{lem:modulation} Let $u(t)\in \mathcal C(I,H^1(\R^d))$ for some interval $I$, be  such that
\begin{equation}\label{close}
\sup_{t\in I}\inf_{\lambda_0>0,\gamma_0}\left\|\lambda_0^{\frac d2} u(t,\lambda_0 y)e^{i\gamma_0}- Q(y) \right\|_{H^1} \leq \delta,
\end{equation}
for $\delta>0$ small enough.
Then, there exist $\mathcal C^1$ functions $\lambda\in(0,+\infty)$, $b\in\R$, $\gamma\in\R$ on $I$ such that $u$ admits 
 a unique decomposition of the form 
\begin{equation}\label{eq:modulation}
u(t,x)=\frac{1}{\lambda^{\frac{d}{2}}(t)}\left(P_{b(t)}+\eps(t,y)\right)e^{i\gamma(t)},\qquad 
y=\frac{x}{\lambda(t)}
\end{equation}
where $\eps$ satisfies  the following orthogonality conditions  on $I$
($\rho_b(t,y)=\rho(y)e^{-i\frac{b(t)|y|^2}{4}}$)
\begin{equation}\label{eq:ortho}
\psld{\eps}{i\Lambda P_b}=\psld{\eps}{|y|^2P_b}=\psld{\eps}{i\rho_b }=0.
\end{equation} 
\end{lemma}
See \eqref{eq:def-P_b} for the definition of $P_b$.

\bigskip

Let $E_0\in \R$. Given $t_1<0$ close to $0$, following Remark \ref{rk:9}, we define the initial rescaled time $s_1$ as 
\begin{equation*}
s_1:=\left|C_s^{-1}t_1\right|^{-\frac{\alpha}{4-\alpha}}.
\end{equation*} 
Let $\l_1$ and $ b_1$ be given by Lemma \ref{le:bu} for this value of $s_1$. Let $u(t)$ be the solution of $\eqref{eq:nls}$ for $t\leq t_1$,  with data 
\begin{equation}\label{eq:final-data}
u(t_1,x)=\frac{1}{\lambda_1^{\frac{d}{2}}}P_{b_1}\left(\frac{x}{\lambda_1}\right).
\end{equation}
As long as the solution $u(t)$ satisfies  \eqref{close}, we consider its decomposition $(\lambda,b,\gamma,\varepsilon)$ from Lemma \ref{lem:modulation} and we define the rescaled time $s$ by 
\be
\label{defst}
s= s_1-\int_{t}^{t_1} \frac{1}{\lambda^2(\tau)}d\tau.
\ee
The heart of the proof of Theorem \ref{thm:1} is the following result, giving uniform backwards estimates on the decomposition of $u(s)$   on $[s_0,s_1]$ for some $s_0$ independent of $s_1$.
 
\begin{proposition}[Uniform estimates in rescaled time]\label{prop:uniform-estimates-rescaled-time}
There exists $s_0>0$ independent of $s_1$ such that
the solution $u$ of $\eqref{eq:nls}$ defined by \eqref{eq:final-data}
exists and  satisfies \eqref{close} on $[s_0,s_1]$. Moreover, its    decomposition
\begin{equation*}
u(s,x)=\frac{1}{\lambda^{\frac{d}{2}}(s)}
\left(P_b+\eps\right)\left(s,y\right) e^{i\gamma(s)},
\qquad 
y=\frac{x}{\lambda(s)}, 
\end{equation*}
satisfies the following uniform estimates on $ [s_0,s_1]$,
\begin{equation}\label{eq:smallness-1-s}
\norm{\eps(s)}_{H^1}\lesssim s^{-(K+1)},\quad
\left|\frac{\lambda^{\frac\alpha 2}(s)}{\lambda_{\rm app}^{\frac\alpha 2}(s)}-1\right|+
\left|\frac{b(s)}{b_{\rm app}(s)}-1\right|  
\lesssim s^{-\frac 12}+s^{2-\frac 4\alpha}.
\end{equation}
In addition, 
\begin{equation*}
|E(P_{b,\lambda,\gamma}(s))-E_0|\leq O(s^{-6}).
\end{equation*}
\end{proposition}
Let us insist again that the key point in Proposition \ref{prop:uniform-estimates-rescaled-time} is that $s_0$ and the constants in the estimates are independent of $s_1\to +\infty$.

\subsection{Proof of Theorem \ref{thm:1} assuming Proposition \ref{prop:uniform-estimates-rescaled-time}}

 First, we convert the estimates of Proposition \ref{prop:uniform-estimates-rescaled-time} in the  original time variable $t$. We claim:
 \begin{lemma}[Estimates in the $t$ variable]\label{le:est-t}
There exists $t_0<0$ such that under the assumptions of  Proposition \ref{prop:uniform-estimates-rescaled-time}, for all $t\in [t_0,t_1]$, 
\bea
&&b(t)= C_b |t|^{\frac\alpha{4-\alpha}}(1+o_{t\uparrow 0}(1))
\label{eq:smallness-1}, \ \ \lambda (t)=C_\lambda|t|^{\frac{2}{4-\alpha}}(1+o_{t\uparrow 0}(1))\\
&&\norm{\eps(t)}_{H^1}\lesssim |t|^{\frac{(K+1)\alpha}{4-\alpha}}\label{eq:smallness-1-e}\\
&& \label{newener}
|E(P_{b,\lambda,\gamma}(t))-E_0|=o_{t\uparrow 0}(1)
\eea
\end{lemma}
\begin{proof}[Proof of Lemma \ref{le:est-t}] Using \eqref{eq:smallness-1-s}, \eqref{defst}, for all large $s<s_1$,
$$
t_1-t(s)=\int_s^{s_1}\l^2(\sigma)d\sigma=\int_s^{s_1}\l^2_{\rm app}(\sigma)\left[1+O(\sigma^{-\frac 12})+O(\sigma^{2-\frac 4\alpha})\right]d\sigma.$$
Recall that $t_{\rm app}$ given by \eqref{deftapp} corresponds to the normalization 
$$t_{\rm app}(s)=-\int_{s}^{+\infty}\l^2_{\rm app}(\sigma),\quad t_{\rm app}(s_1)=t_1,$$
from which we obtain
$$t(s)=t_{\rm app}(s)(1+o(1))=-C_s s^{-\frac{4-\alpha}{\alpha}}\left[1+o(1)\right].$$ The estimates of Lemma \ref{le:est-t} now follow directly follow from \eqref{eq:dbu} and Proposition \ref{prop:uniform-estimates-rescaled-time}  (see the definition of $C_\lambda$ and $C_b$ in \eqref{Cbumille} and \eqref{Cbu2}).
\end{proof}

Now, we finish the proof of Theorem \ref{thm:1} assuming Proposition \ref{prop:uniform-estimates-rescaled-time}.

\begin{proof}[Proof of Theorem \ref{thm:1}]
Let $(t_n)\subset(t_0,0)$ be an increasing sequence  such that $\lim_{n\to \infty}t_n=0$. For each $n$, let $u_n$  be the solution of \eqref{eq:nls} on $[t_0,t_n]$ with final data at $t_n$ 
\begin{equation}\label{eq:final-data-n}
u_n(t_n,x)=\frac{1}{\lambda^{\frac{d}{2}}(t_n)} P_{b(t_n)}\left(\frac{x}{\l(t_n)}\right),
\end{equation}
where $\lambda(t_n)=\lambda_1$ and  $b(t_n)=b_1$ are given by  Lemma \ref{le:bu}
for $s_1=|C_s^{-1}t_n|^{-\frac{\alpha}{4-\alpha}}$, 
so that $u_n(t)$ satisfies the conclusions of Proposition \ref{prop:uniform-estimates-rescaled-time} and of Lemma \ref{le:est-t} on the interval $[t_0,t_n]$.
The minimal mass blow up solution for \eqref{eq:nls} is now obtained as the limit of a subsequence of $(u_n)$. In a first step, we prove that a subsequence  of $(u_n(t_0))$ converges to a suitable initial data. Indeed, from Lemma \ref{le:est-t}, we infer that $(u_n(t_0))$ is bounded in $H^1(\R^d)$.  Hence there exists a subsequence of $(u_n(t_0))$ (still denoted by $(u_n(t_0))$ and 
 $u_\infty(t_0)\in H^1(\R^d)$ such that 
\[
u_n(t_0)\rightharpoonup  u_\infty(t_0)\quad\text{weakly in }H^1(\R^d)\text{ as }n\to+\infty.
\]
Now, we obtain   strong  convergence in $H^s$ (for some $0<s<1$) by direct arguments.  Let $\chi:[0,+\infty)\to[0,1]$ be a smooth cut-off function such that $\chi\equiv 0$ on  $[0,1]$ and $\chi\equiv 1$ on $[2,+\infty)$. 
For $R>0$, define $\chi_R:\R^d\to[0,1]$ by $\chi_R(x)=\chi(|x|/R)$. 
Take any $\delta>0$. By the expression of $u_n(t_n)$ in \eqref{eq:final-data-n}, we can choose $R$ large enough (independent of $n$) so that  
\begin{equation}\label{n-delta}
\int_{\R^d} |u_n(t_n)|^2\chi_R dx\leq \delta.
\end{equation}
It follows from elementary computations that  
\[
\frac{d}{d t}\int_{\R^d} |u_n|^2\chi_R dx=2\, \Im\int_{\R^d}\nabla\chi_R\cdot \nabla u_n \, \bar u_n dx.
 \]
Hence from the geometrical decomposition 
\begin{equation*}
u_n(t,x)=\frac{1}{\lambda_n^{\frac{d}{2}}(t)}\left(P_{b_n(t)}+\eps_n)(t,y)\right)e^{i\gamma_n(t)},\qquad 
y=\frac{x}{\lambda_n(t)},
\end{equation*}
 and the smallness \eqref{eq:smallness-1}-\eqref{eq:smallness-1-e} of $\eps_n$  and $\lambda_n$ we infer
\[
\left|\frac{d}{d t}\int_{\R^d} |u_n(t)|^2\chi_R dx\right|\leq  \frac{C}{\lambda_n(t) R}\left(e^{-\frac{R}{2\lambda_n(t)}}+\norm{\eps_n(t)}^2_{H^1}\right)\leq \frac{C}{R}|t|^{\left(-\frac{2}{\alpha}+K+1\right)\frac{\alpha}{4-\alpha}}.
 \]
Integrating between $t_0$ and $t_n$, we obtain 
\[
\int_{\R^d} |u_n(t_0)|^2\chi_R dx\leq \frac{C}{R}|t_0|^{\left(-\frac{2}{\alpha}+K+1\right)\frac{\alpha}{4-\alpha}+1}+\int_{\R^d}|u_n(t_n)|^2\chi_R dx.
\]
Combined with \eqref{n-delta}, for a possibly larger $R$, this implies
\[
\int_{\R^d} |u_n(t_0)|^2\chi_R dx\leq 2\delta.
\]
We conclude from the local compactness of Sobolev embeddings that for $0\leq s<1$:
\[
u_n(t_0)\to  u_\infty(t_0)\quad\text{strongly in }H^s(\R^d),\ \text{ as }n\to+\infty.
\]
Let $u_\infty(t)$ be the solution of \eqref{eq:nls} with $u_\infty(t_0)$ as initial data at $t=t_0$. From \cite{Ca03,CaWe90} there exists $0<s_0<1$ such that the Cauchy problem for \eqref{eq:nls} is locally well-posed  in $H^{s_0}(\R^d)$. This implies that $u_\infty$ exists on $[t_0,0)$ and for any $t\in[t_0,0)$,  
\[
u_n(t)\to  u_\infty(t)\quad\text{strongly in }H^{s_0}(\R^d),\,\text{ weakly in }H^1(\R^d),\,\text{ as }n\to+\infty.
\]
Moreover,
since $\lim_{n\to \infty} \int u_n^2(t_n) = \int Q^2$, we have $\int u_\infty^2=\int Q^2$. By weak convergence in $H^1(\R^d)$ and the estimates from Lemma \ref{le:est-t} applied to $u_n$, $u_\infty(t)$ satisfies \eqref{close}, and denoting $(\eps_\infty,\lambda_\infty,b_\infty,\gamma_\infty)$ its decomposition, we have by standard arguments (see e.g. \cite{MeRa05}), for any $t\in[t_0,0)$,
$$
\lambda_n(t)\to \lambda_\infty(t), \quad
b_n(t)\to b_\infty(t), \quad
\gamma_n(t)\to \gamma_\infty(t),
\quad \eps_n(t)\rightharpoonup \eps_\infty(t) \quad \hbox{$H^1(\R^d)$ weak, as $n\to \infty$.}
$$ 
The uniform estimates on $u_n$ from Lemma \ref{le:est-t} give, on $[t_0,0)$,
\be
\label{eq:smallness-inf-e}
b_\infty(t)=C_b|t|^{\frac \alpha{4-\alpha}} \left(1+o_{t\uparrow 0}(1)\right) ,\quad \lambda_\infty (t)=C_\lambda |t|^{\frac {2 }{4-\alpha}} \left(1+o_{t\uparrow 0}(1)\right) , \ \ \|\varepsilon_{\infty}(t)\|_{H^1}\lesssim |t|^{\frac{(K+1)\alpha}{4-\alpha}},
\ee
\begin{equation}\label{eq:blambda2}
\frac{b_\infty(t)}{\lambda_\infty^2(t)}=\frac{C_b}{C_\lambda^2}|t|^{\frac \alpha{4-\alpha} - \frac{4}{4-\alpha}}
\left(1+o_{t\uparrow 0}(1)\right)
=\frac{2}{4-\alpha}\frac{1}{|t|}\left(1+o_{t\uparrow 0}(1)\right)=\frac{\sigma}{|t|}\left(1+o_{t\uparrow 0}(1)\right),
\end{equation}
which justifies the form \eqref{eq:form} and the blow up rate \eqref{blowupspeed}.
Finally, we prove that $E(u_\infty)=E_0$.
Let $t_0<t<0$. We have by \eqref{newener} and \eqref{eener},
$$
\mathcal{E}(b_n(t),\lambda_n(t))-\frac{8E_0}{\int |y|^2 Q^2}=o_{t\uparrow 0}(1)
$$
where the $o_{t\uparrow 0}(1)$ is independent of $n$,
and thus
$$
\mathcal{E}(b_\infty(t),\lambda_\infty(t))-\frac{8E_0}{\int |y|^2 Q^2}=o_{t\uparrow 0}(1)
$$
Using \eqref{eener}, we deduce
$$
E(P_{b_\infty,\lambda_\infty,\gamma_\infty}(t))-E_0=o_{t\uparrow 0}(1)$$
and thus, by \eqref{eq:smallness-inf-e},
$$
E(u_\infty(t))-E_0=o_{t\uparrow 0}(1).
$$
Thus, by conservation of energy and passing to the limit $t\uparrow0$, we obtain $E(u_\infty(t))=E_0$.
\end{proof}

\subsection{Bootstrap estimates}\label{sec:4:3}
The rest of the paper is devoted to the proof of Proposition \ref{prop:uniform-estimates-rescaled-time}.
We use  a bootstrap argument involving  the following estimates:
\begin{equation}
\label{eq:smallness-4-s}\norm{\eps(s)}_{H^1} < s^{-K},\quad
\left|\frac{\lambda^{\frac\alpha2}(s)}{\lambda_{\rm app}^{\frac\alpha2}(s)}-1\right|
+\left|\frac{b(s)}{b_{\rm app}(s)}-1\right|  <s^{-\delta(\alpha)}
\end{equation}
for some small enough universal constant $\delta(\alpha)>0$.
The following value is suitable in this paper
\begin{equation}\label{deltaalpha}
	\delta(\alpha)= \min\left(\frac 14,\frac 2\alpha-1\right)>0. 
\end{equation}
For $s_0>0$  to be chosen large enough (independently of $s_1$),   we 
define
\begin{equation}\label{def:st}
s_*=\inf\{ \tau\in[s_0,s_1];\eqref{eq:smallness-4-s} \text{ holds on }  [\tau,s_1]\}.
\end{equation}
Observe from \eqref{app_ini} that 
$$\left|\frac{\lambda_1^{\frac\alpha2}}{\lambda_{\rm app}^{\frac\alpha2}(s_1)}-1\right|+\left|\frac {b_1}{b_{\rm app}(s_1)}-1\right|\lesssim s_1^{-\frac 12}+s_1^{2-\frac{4}{\alpha}}\ll s_1^{-\delta(\alpha)},$$  for $s_1$ large, and hence by the definition \eqref{eq:final-data} of  $u(s_1)$, $s_*$ is well-defined and $s_*<s_1$. 
In \S5, \S6 and \S7,  we prove that \eqref{eq:smallness-1-s} holds on $[s_*,s_1]$.
By a standard continuity argument, provided that $s_0$ is large enough, we obtain $s_*=s_0$ which implies Proposition \ref{prop:uniform-estimates-rescaled-time}. The main lines  of the proof are as follows: first,  we derive modulation equations  from the construction of $P_b$,
second we control the remaining error using a mixed Energy/Morawetz functional first derived in \cite{RaSz11}.

\section{Modulation equations}

In this section, we work with the solution $u(t)$ of Proposition \ref{prop:uniform-estimates-rescaled-time} on the time interval $[s_*,s_1]$ (see \eqref{eq:smallness-4-s}-\eqref{def:st}). We justify that the  dynamical system satisfied by the modulation parameters $\lambda,b$ is at the main order given by \eqref{eq:sys}. 
Define \[
\Mod(s) =\begin{pmatrix}
b+\frac{\lambda_s}{\lambda}\\
b_s+b^2-\theta\\
1-\gamma_s
\end{pmatrix}.
\]

\begin{lemma}[Modulation equations and additional orthogonality]\label{lem:estimate-modulation-equations}
For all $s\in [s_*,s_1]$,
\begin{equation}\label{eq:estimate-modulation}
|\Mod(s)|\lesssim \frac 1{s^{K+2}},
\end{equation} 
\begin{equation}\label{eq:additional-ortho}
|\psld{\eps(s)}{Q}|\lesssim  \frac 1{s^{K+1}}.
\end{equation}
\end{lemma}
\begin{proof}[Proof of Lemma \ref{lem:estimate-modulation-equations}]
The proofs of the two estimates are combined.
Since $\eps(s_1)\equiv0$, we may define
$$
s_{**}=\inf\{s\in [s_*,s_1];\   |\psld{\eps(\tau)}{P_b}|<\tau^{-(K+2)}\text{ holds on }  [s,s_1]\}.
$$ 
We work on the interval $[s_{**},s_1]$.

Since $P_b$ verifies equation \eqref{eq:P_b}, we obtain the following equation for $\eps$:
\begin{multline}\label{eq:eps}
i\eps_s+\Delta \eps-\eps+ib\Lambda\eps+(f(P_b+\eps)-f(P_b))+\lambda^\alpha (g(P_b+\eps)-g(P_b))\\
-i\left(b+\frac{\lambda_s}{\lambda}\right)\Lambda (P_b+\eps)
+(1-\gamma_s)(P_b+\eps)
+(b_s+b^2-\theta)\frac{|y|^2}{4}P_b\\
=-\Psi e^{-i\frac{b|y|^2}{4}}.
\end{multline}
where $\Psi:=\Psi_K$. Recall  that equation \eqref{eq:eps} combined with the orthogonality conditions chosen on $\eps$ -- see \eqref{eq:ortho} -- contains the equations of the modulation parameters. Technically, one differentiates in time the orthogonality conditions for $\eps$, then uses the equation \eqref{eq:eps} on $\eps$ and the estimate \eqref{eq:error-term-estimate} on the error term $\Psi$.
Here, as in \cite{RaSz11}, the orthogonality conditions are chosen to obtain quadratic control in $\eps$. Since it is a   standard argument  (see e.g. \cite{MeRa06,PlRa07,RaSz11}), we only sketch   relevant computations. 

Consider for example the orthogonality condition $\psld{\eps}{i\Lambda P_b}=0$. Differentiating in $s$, we obtain 
$
\dual{\eps_s}{i\Lambda P_b}+ \dual{\eps}{i \partial_s(\Lambda P_b)}=0.
$
Since
$$
\frac d{d s} ( \Lambda P_b) 
  = \left( (\Lambda P)_s   - i \frac {b_s}{4} |y|^2 \Lambda P \right) e^{-i \frac b4 |y|^2},
$$
and
\begin{multline*}
(\Lambda P)_s  = \lambda^\alpha 
\Bigg( 
\alpha \frac{\lambda_s}{\lambda} \bigg(Z+\sum_{(j,k)\in\Sigma_K}kb^{2j}\lambda^{k\alpha-1}(P_{j,k}^++bP_{j,k}^-)\bigg)\\
+b_s\bigg(\sum_{(j,k)\in\Sigma_K}2jb^{2j-1}\lambda^{k\alpha}P_{j,k}^+  +\sum_{(j,k)\in\Sigma_K}(2j+1)b^{2j}\lambda^{k\alpha}P_{j,k}^-\bigg)
 \Bigg),
\end{multline*}
proceeding as in the proof of Proposition \ref{prop:profile}, and using the properties of the functions $P_{j,k}^\pm$, we  note that 
\begin{equation*}
\sup_{y\in \R}\left(  e^{\frac y2} \left| \frac d{d s} ( \Lambda P_b) (y)\right| \right)\lesssim |\Mod(s)|+b^2(s) + \lambda^\alpha(s).
\end{equation*}
Thus,  by \eqref{eq:smallness-4-s},
\[
  |\psld{\eps}{i \partial_s(\Lambda P_b)}| \lesssim \|\eps(s)\|_{2}\left(|\Mod(s)|+b^2(s) + \lambda^\alpha(s)\right)
  \nonumber \\ 
  \lesssim s^{-2}  |\Mod(s)|+ s^{-(K+2)}.
\]
Next, we write $\dual{\eps_s}{i\Lambda P_b}=-\dual{i\eps_s}{\Lambda P_b}$ and we use the equation of $\eps$.
We start by the contribution of the first line of \eqref{eq:eps}. 
Remark that by \eqref{eq:smallness-4-s},
\begin{align*}
f(P_b+\eps)-f(P_b)&=e^{-ib\frac{|y|^2}{4}}\left( f\left(P+e^{ib\frac{|y|^2}{4}}\eps\right)-f(P)\right)=e^{-ib\frac{|y|^2}{4}}df(P)\left(e^{ib\frac{|y|^2}{4}}\eps\right)+O(|\eps|^2)\\
&=e^{-ib\frac{|y|^2}{4}}df(P)\left(e^{ib\frac{|y|^2}{4}}\eps\right)+O(s^{-2}|\eps| ),
\end{align*}
\[
\lambda^{\alpha} \left( g (P_b+\eps)-g(P_b)\right) =  O(\lambda^\alpha |\eps|)
= O(s^{-2}|\eps|),
\]
and  
\[
\Delta\eps+ib\Lambda\eps=e^{-ib\frac{|y|^2}{4}}\Delta \left(e^{ib\frac{|y|^2}{4}} \eps\right)+b^2\frac{|y|^2}{4}\eps, \qquad
\Lambda P_b=e^{-ib\frac{|y|^2}{4}}\left(\Lambda P-ib\frac{|y|^2}{2}P\right).
\]
Therefore,   using  \eqref{eq:smallness-4-s} and $P=Q+O_{H^1}(s^{-2})$ (see  the definition of $P$ in \eqref{eq:def-P}), we have
\begin{multline*}
\dual{-\Delta\eps+\eps-ib\Lambda\eps-(f(P_b+\eps)-f(P_b))+\lambda^{\alpha} \left( g (P_b+\eps)-g(P_b)\right)}{\Lambda P_b}\\
=\dual{-\Delta \left(e^{ib\frac{|y|^2}{4}} \eps\right)+e^{ib\frac{|y|^2}{4}}\eps-pQ^{p-1}\left(e^{ib\frac{|y|^2}{4}}\eps\right)}{\Lambda Q-ib\frac{|y|^2}{2}Q}+O(s^{-2}\|\eps\|_{2})\\
=\dual{L_+ \left(e^{ib\frac{|y|^2}{4}} \eps\right)}{\Lambda Q}-\frac b2\dual{L_-\left(e^{ib\frac{|y|^2}{4}} \eps\right)}{i|y|^2Q}+O(s^{-2}\|\eps\|_{2})\\
=\dual{e^{ib\frac{|y|^2}{4}} \eps}{L_+(\Lambda Q)}-\frac b2\dual{e^{ib\frac{|y|^2}{4}} \eps}{iL_-(|y|^2Q)}+O(s^{-2}\|\eps\|_{2})\\
=-2\psld{\eps}{ e^{-ib\frac{|y|^2}{4}}Q}+2b\psld{\eps}{ie^{-ib\frac{|y|^2}{4}}\Lambda Q}+O(s^{-2}\|\eps\|_{2})\\
=-2\psld{\eps}{ P_b}+2b\psld{\eps}{i \Lambda P_b}+O(s^{-2}\|\eps\|_{2})
=O(s^{-(K+2)}).
\end{multline*}
Note that we have used  algebraic relations  from \eqref{algebra}, then \eqref{eq:smallness-4-s},  $\psld{\eps}{i\Lambda P_b}=0$ and the definition of $s_{**}$.

The part  corresponding to the second line of \eqref{eq:eps} gives 
\begin{multline*}
\psld{-i\left(b+\frac{\lambda_s}{\lambda}\right)\Lambda(P_b+\eps)+(1-\gamma_s)(P_b+\eps)+(b_s+b^2-\theta)\frac{|y|^2}{4}P_b}{\Lambda P_b}\\
=-(b_s+b^2-\theta)\norm{yP_b}_2^2+O(|\Mod(s)|\norm{\eps}_2)\\
=-(b_s+b^2-\theta)(\norm{yQ}_2^2+O\left(s^{-2})\right)+O(s^{-2}|\Mod(s)|).
\end{multline*}

Finally, from the estimate \eqref{eq:error-term-estimate} on $\Psi$, we have
$$
\left| \psld{\Psi}{\Lambda P-ib\frac{|y|^2}{2}P}\right|\lesssim s^{-2} |\Mod(s)| + s^{-2(K+2)}.
$$ 
Combining the previous estimates, we find
\begin{equation*}
|b_s+b^2-\theta| \lesssim s^{-2} |\Mod(s)| + s^{-(K+2)}.
\end{equation*}

Using the other orthogonality conditions in \eqref{eq:ortho} in a similar way, together with \eqref{algebra}, we find
\begin{equation*}
 |\Mod(s)|
\lesssim  s^{-2} |\Mod(s)|  + s^{-(K+2)}. 
\end{equation*}
We deduce that for all $s\in [s_{**},s_1]$,
\begin{equation}\label{onmod}
|\Mod(s)|\lesssim s^{-(K+2)}.
\end{equation}

By conservation of the $L^2$ norm and \eqref{eq:final-data}, we have
\[
\norm{u(s)}_2^2=\norm{u(s_1)}_2^2=\norm{P_b(s_1)}_2^2.
\]
Thus, by \eqref{eq:modulation},  
\[
\psld{\eps(s)}{P_b}=\frac12\left(\norm{u(s)}_2^2-\norm{P_b(s)}_2^2-\norm{\eps(s)}_2^2\right)=
-\frac 12 \norm{\eps(s)}_2^2+\frac 12 \left( \norm{P_b(s_1)}_2^2-\norm{P_b(s)}_2^2\right).
\]
Moreover, by  \eqref{dmass}, \eqref{eq:smallness-4-s} and \eqref{onmod},
\[
\frac {d}{ds}\int |P_b|^2 
\lesssim s^{-(K+4)}. 
\]
Integrating and combining the previous estimates with \eqref{eq:smallness-4-s},  we obtain, for all $s\in [s_{**},s_1]$,  
\begin{equation}\label{eq:new-ortho-1}
|\psld{\eps(s)}{P_b}|\lesssim s^{-(K+3)}.
\end{equation}
Therefore, $s_{**}=s_*$ and the estimates \eqref{onmod} and \eqref{eq:new-ortho-1} are proved on 
$[s_*,s_1]$. Since $|P_b - Q|\lesssim  Q^{\frac 12} s^{-1}$, we obtain \eqref{eq:additional-ortho}.
\end{proof}

\section{The mixed energy Morawetz monotonicity formula}

In this section, following \cite{RaSz11}, we introduce a mixed Energy/Morawetz functional to control the remaining part of the solution in $H^1(\R^d)$. First, define the   energy of $\eps$
\begin{multline*}
H(s,\eps):=\frac12\norm{\nabla\eps}_2^2+\frac12\norm{\eps}_2^2-\int_{\R^d} (F(P_b+\eps)-F(P_b)-dF(P_b)\eps)dy\\
-\lambda^\alpha \int_{\R^d} (G(P_b+\eps)-G(P_b)-dG(P_b)\eps)dy.
\end{multline*}
Note that as in \cite{RaSz11}, the time derivative of the linearized energy $H$ for $\eps$ cannot be controlled alone, and one has to add a virial type functional  such as 
$
\frac b2\Im\int_{\R^d} \nabla\left(\frac{|y|^2}{2}\right) \nabla \eps\bar\eps dy.
$
In practice, due to the lack of control on $\norm{y\eps}_2$, we replace $\frac 12 |y|^2$ by a   function whose gradient is bounded, which we  introduce now.\\
Let $\phi:\R\to \R$ be a smooth even and convex function, nondecreasing on $\R^+$,
such that   
\[
\phi(r)=
\left\{
\begin{aligned}
&\frac12r^2& \text{ for }&r<1,\\
&3r+e^{-r}& \text{ for }&r>2,
\end{aligned}
\right.
\]
and set $\phi(x)=\phi(|x|)$.
Let $A\gg1$ to be fixed. Define $\phi_A$ by $\phi_A(y)=A^2\phi\left(\frac{y}{A}\right)$ and 
\[
J(\eps)=\frac12\Im\int_{\R^d}\nabla\phi_A\cdot \nabla \eps\bar\eps dy.
\]
Finally, set
\[ 
S(s,\eps)=\frac{1}{\lambda^{4}(s)}(H(s,\eps)+b(s)J(\eps(s))).
\] 

The relevance of the functional $S$ lies on  the following two   properties. 

\begin{proposition}[Coercivity of $S$]\label{prop:coercivity-of-S}
For any $s\in[s_*,s_1]$,
\[
S(s,\eps(s))\gtrsim \frac{1}{\lambda^{4}(s)}\left(
\norm{\eps(s)}_{H^1}^2+O(s^{-2(K+1)}\right).
\]
\end{proposition}

\begin{proposition}\label{prop:Lyapunov}
For any $s\in[s_*,s_1]$,
\[
\frac{d}{d s}\left[ S(s,\eps(s))\right] \gtrsim 
\frac{b}{\lambda^{4}(s)}\left(\norm{\eps(s)}_{H^1}^2+O(s^{-2(K+1)}\right).
\]
\end{proposition}
 
The rest of this section is organized as follows. We first prove  Proposition \ref{prop:coercivity-of-S} in \S \ref{subsec:coercivity}. 
In \S \ref{subsec:time-derivative-H} we compute the time derivative of $H$ and in \S \ref{subsec:time-derivative-J}, the time derivative of $J$.
We finish the proof of Proposition \ref{prop:Lyapunov} in \S \ref{subsec:lyapunov}.

\subsection{\texorpdfstring{Coercivity of $S$}{Coercivity of S}}\label{subsec:coercivity}
We prove Proposition \ref{prop:coercivity-of-S}. We first claim a  coercivity property for $H$,   consequence of the   properties of $L_+$ and $L_-$ (see \eqref{eq:coercivity}) and of the orthogonality conditions of $\eps$ (see \eqref{eq:ortho}).

\begin{lemma}[Coercivity of $H$]\label{lem:coercivity-of-H}
For all $s\in[s_*,s_1]$, 
\[
H(s,\eps)\gtrsim 
\norm{\eps}_{H^1}^2+O(s^{-2(K+1)}).
\]
\end{lemma}

\begin{proof}
From the orthogonality conditions \eqref{eq:ortho}, \eqref{eq:additional-ortho}, and  estimates \eqref{eq:smallness-4-s}, the following holds:
\begin{align*}
\psld{\eps}{|y|^2 Q}&=\psld{\eps}{|y|^2P_b}+O(|b| \norm{\eps}_2)+O(\lambda^\alpha\norm{\eps}_2) =O(s^{-1}\norm{\eps}_{H^1}) ,\\
\psld{\eps}{i\rho}  &  =\psld{\eps}{i\rho_b} +O(|b|\norm{\eps}_2)=O(s^{-1}\norm{\eps}_{H^1}),\\
\psld{\eps}{ Q}&        =O(s^{-(K+1)}).
\end{align*}

From \eqref{eq:smallness-4-s}, we have 
\begin{equation*}
\lambda^\alpha \int_{\R^d} (G(P_b+\eps)-G(P_b)-dG(P_b)\eps)dx= O(s^{-2}\norm{\eps}_{H^1}^2).
\end{equation*}
Next, (denoting $\eps = \eps_1+ i\eps_2$),
$$\left|  F(P_b+\eps)-F(P_b)-dF(P_b)\eps 
- \left(1+\frac{4}d\right) Q^{\frac{4}{d}} \eps_1^2 - \frac 12 Q^{\frac{4}{d}} \eps_2^2\right|
  \lesssim e^{-\frac 12 |y|} |\eps|^3 +|\eps|^{2+\frac 4d}+ |\eps|^2 (|b|+\lambda^\alpha).
$$
Thus, from \eqref{eq:smallness-4-s}, 
$$\left| \int F(P_b+\eps)-F(P_b)-dF(P_b)\eps 
- \left(1+\frac{4}d\right) Q^{\frac{4}{d}} \eps_1^2 - \frac 12Q^{\frac{4}{d}} \eps_2^2\right|
  \lesssim 
 O(s^{-1}\norm{\eps}_{H^1}^2),
$$
and
$$
\left|H(s,\eps) - \frac 12 \dual{L_+\eps_1}{\eps_1}-\frac 12 \dual{L_-\eps_2}{\eps_2}\right|\lesssim 
O(s^{-1}\norm{\eps}_{H^1}^2) .
$$
Combining these estimates with the coercivity properties of   $L_+$, $L_-$ (see \eqref{eq:coercivity}),
we obtain the result.
\end{proof}

Since 
\[
|bJ(\eps)|\leq |b|\norm{\nabla\phi_A}_{\infty}\norm{\eps}^2_{H^1}\lesssim O(s^{-1}\norm{\eps}_{H^1}^2)
\]
(from \eqref{eq:smallness-4-s}), Lemma \ref{lem:coercivity-of-H} implies
Proposition \ref{prop:coercivity-of-S}. 

\bigskip

For future reference, we also claim the following localized coercivity property (see similar statement in \cite{MaMe02} and \cite{RaSz11}).
\begin{lemma}\label{le:9bis}
There exists $A_0>1$ such that for any $A>A_0$,
\[
\frac12\int_{\R^d}\phantom{}\nabla\eps^T   \nabla^2\phi_A\nabla\bar\eps dy+\frac12\norm{\eps}_2^2-\int_{\R^d} (F(P_b+\eps)-F(P_b)-dF(P_b)\eps)dx\gtrsim \norm{\eps}^2_{2}+O(s^{-2(K+1)}).
\]
\end{lemma}
For now on, we consider $A>A_0$.

\subsection{\texorpdfstring{Time variation of the energy of $\eps$}{Time variation of the energy of ε} }\label{subsec:time-derivative-H}
\begin{lemma}\label{lem:derivative-H} For all $s\in[s_*,s_1]$, 
\[
\frac{d}{d s}[H(s,\eps(s))]=
\frac{\lambda_s}{\lambda}\left(\norm{\nabla\eps}_2^2-\dual{f(P_b+\eps)-f(P_b)}{\Lambda\eps}\right)
+O(s^{-(2K+3)})+O(s^{-2}\norm{\eps}_{H^1}^2).
\]
\end{lemma}
 
\begin{proof}[Proof of Lemma \ref{lem:derivative-H}]
The  time derivative for $H$ separates into two parts:
\[
\frac{d}{d s}[H(s,\eps(s))]= D_s H(s,\eps) + \dual{D_\eps H(s,\eps)}{\eps_s},
\]
where $D_s$ (respectively, $D_\eps$) denotes differentiation of the functional with respect to $s$ (respectively, $\eps$). In particular,
\begin{align*}
D_s H(s,\eps) & = - \int (P_b)_s \left( f(P_b+\eps) - f(P_b) - df(P_b)\eps \right) 
-\lambda^\alpha \int (P_b)_s \left( g(P_b+\eps)-g(P_b) - dg(P_b)\eps\right) \\
&- \alpha \frac {\lambda_s}{\lambda} \lambda^\alpha \int \left( G(P_b+\eps) - G(P_b) - dG(P_b)\eps\right).
\end{align*}
Note that 
$$e^{ i \frac {b |y|^2}{4}} (P_b)_s  =  P_s  - i b_s \frac {|y|^2}{4} P 
=  P_s - i \left(b_s +b^2 - \beta \lambda^\alpha\right) \frac {|y|^2}{4} P +i\left(b^2 - \beta \lambda^\alpha\right) \frac {|y|^2}{4} P.
$$
By \eqref{Ps},  \eqref{eq:smallness-4-s} and  Lemma \ref{lem:estimate-modulation-equations}, we obtain
$$
|(P_b)_s | \lesssim s^{-2} e^{-\frac{|y|}{2}} \quad \hbox{and} \quad \left|\frac {\lambda_s}{\lambda}\right|\lambda^\alpha \lesssim s^{-3}.
$$
Thus,
\begin{equation*}
| D_s H(s,\eps) | \lesssim s^{-2} \|\eps\|_{H^1}^2.
\end{equation*}

Now, we compute $\dual{D_\eps H(s,\eps)}{\eps_s}$.
Note that  \eqref{eq:eps} rewrites
\begin{equation}\label{47bis}
i\eps_s-{D_\eps} H(s,\eps)+\ModOp(s)P_b-i\frac{\lambda_s}{\lambda} \Lambda\eps
+(1-\gamma_s)\eps
+e^{-ib\frac{|y|^2}{4}}\Psi=0,
\end{equation}
where  
\[
\ModOp(s)P_b:=-i\left(b+\frac{\lambda_s}{\lambda}\right)\Lambda P_b
+(1-\gamma_s)P_b
+(b_s+b^2-\theta)\frac{|y|^2}{4}P_b.
\]
Using  \eqref{47bis}, since $\dual{i{D_\eps} H(s,\eps)}{{D_\eps}H(s,\eps)}=0$, we have
 \begin{multline}\label{eq:to-be-estimated}
\dual{{D_\eps} H(s,\eps)}{\eps_s}=\dual{i{D_\eps} H(s,\eps)}{i\eps_s}=
-\dual{i{D_\eps} H(s,\eps)}{\ModOp(s)P_b}+
\frac{\lambda_s}{\lambda} \dual{i{D_\eps} H(s,\eps)}{i\Lambda\eps}\\
-(1-\gamma_s)\dual{i{D_\eps} H(s,\eps)}{\eps}
-\dual{i{D_\eps} H(s,\eps)}{e^{-ib\frac{|y|^2}{4}}\Psi}.
\end{multline}
From  the proof of Lemma \ref{lem:estimate-modulation-equations}   
\begin{multline*}
{D_\eps} H(s,\eps)=
-\Delta \eps+\eps-(f(P_b+\eps)-f(P_b))-\lambda^{\alpha}(g(P_b+\eps)-g(P_b))
\\
=e^{-ib\frac{|y|^2}{4}}\left(L_+\Re\left(e^{ib\frac{|y|^2}{4}}\eps\right)+iL_-\Im\left(e^{ib\frac{|y|^2}{4}}\eps\right)\right)+ib\Lambda\eps+b^2\frac{|y|^2}{4}\eps+ O(s^{-2} |\eps|).
\end{multline*}
Therefore, using the orthogonality conditions \eqref{eq:ortho}, \eqref{eq:additional-ortho} and   estimates \eqref{eq:smallness-4-s}, we have (see also proof of Lemma \ref{lem:estimate-modulation-equations}),
\begin{equation*}
\dual{{D_\eps} H(s,\eps)}{\Lambda P_b}=-2\psld{\eps}{P_b}+b\psld{\eps}{i\Lambda P_b}+ O(s^{-2}\norm{\eps}_2)=O(s^{-(K+1)}).
\end{equation*}
Thus, from Lemma \ref{lem:estimate-modulation-equations},
$$\left|\frac {\lambda_s}{\lambda} + b\right| \left|\dual{{D_\eps} H(s,\eps)}{\Lambda P_b}\right|\lesssim
O(s^{-(2K+3)}).
$$

Using similar arguments we get
$$\dual{{D_\eps} H(s,\eps)}{ iP_b} 
 =-4\psld{\eps}{\Lambda P_b}+O(s^{-1}\norm{\eps}_2)=O(s^{-1}\norm{\eps}_2)=O(s^{-(K+1)})
$$
and  
$$
\dual{{D_\eps} H(s,\eps)}{ i\frac{|y|^2}{4}P_b}=
\psld{\eps}{\rho_b}+O(s^{-1}\norm{\eps}_2)=O(s^{-1}\norm{\eps}_2)=O(s^{-(K+1)}).
$$
Using Lemma \ref{lem:estimate-modulation-equations}, we obtain
  in conclusion for this term
\[
\dual{i{D_\eps} H(s,\eps)}{\ModOp(s)P_b}=O(s^{-(2K+3)}).
\]

Next, we have
\begin{multline*}
\dual{i{D_\eps} H(s,\eps)}{i\Lambda\eps}=\dual{{D_\eps} H(s,\eps)}{\Lambda\eps}=\\
\dual{-\Delta \eps+\eps-(f(P_b+\eps)-f(P_b))-\lambda^\alpha (g(P_b+\eps)-g(P_b))}{\Lambda\eps}.
\end{multline*}
Note that (by direct computations)
\[
\dual{-\Delta \eps}{\Lambda\eps}=\norm{\nabla\eps}_2^2,\quad \dual{ \eps}{\Lambda\eps}=0,
\]
and by \eqref{eq:smallness-4-s},
$$
\left|\dual{\lambda^\alpha (g(P_b+\eps)-g(P_b))}{\Lambda\eps}\right|\lesssim O(s^{-2}\norm{\eps}^2_{H^1}).
$$
Thus,
$$
\frac{\lambda_s}{\lambda} \dual{i{D_\eps} H(s,\eps)}{i\Lambda\eps}=
\frac{\lambda_s}{\lambda}\left(\norm{\nabla\eps}_2^2-\dual{f(P_b+\eps)-f(P_b)}{\Lambda\eps}\right)+O(s^{-3} \|\eps\|_{H^1}^2).
$$
 
For the third term in the right-hand side of \eqref{eq:to-be-estimated}, we claim
\begin{multline*}
|(1-\gamma_s)\dual{i{D_\eps} H(s,\eps)}{\eps}|=\Big|(1-\gamma_s)\dual{(f(P_b+\eps)-f(P_b))+\lambda^\alpha (g(P_b+\eps)-g(P_b))}{\eps}\Big|\\
\lesssim
|\Mod(s)|\Big(
\norm{\eps}_2^2+\norm{\eps}_{H^1}^{2+\frac4d}\Big)=O(s^{-4}\norm{\eps}_{H^1}^2).
\end{multline*}

Finally, the fourth term in the right-hand side of  \eqref{eq:to-be-estimated} 
 is estimated by  \eqref{eq:error-term-estimate} combined with Lemma~\ref{lem:estimate-modulation-equations} and    \eqref{eq:smallness-4-s}
 \[
|\dual{i{D_\eps} H(s,\eps)}{\Psi}|\leq O(s^{-(K+4)}\norm{\eps}_{H^1})
\leq O(s^{-(2K+3)})+O(s^{-5}\norm{\eps}_{H^1}^2).
\]
Gathering these estimates, we have proved the lemma.
 \end{proof}

\subsection{The time derivative of the Morawetz part}\label{subsec:time-derivative-J}

\begin{lemma}\label{lem:derivative-J}
For all $s\in[s_*,s_1]$, 
\begin{multline*}
\frac{d}{d s}[J(\eps(s))]=\int_{\R^d}\phantom{}\nabla\eps^T   \nabla^2\phi_A\nabla\bar\eps dy
-\frac14\int_{\R^d}|\eps|^2\Delta^2\phi_Ady\\
-\dual{f(P_b+\eps)-f(P_b)}{\frac12\Delta\phi_A\eps+\nabla\phi_A\nabla \eps}
+O(s^{-(2K+2)})+O(s^{-2}\norm{\eps}_{H^1}^2).
\end{multline*}
\end{lemma}

\begin{proof}
From the definition of $J(\eps)$, we have
\[
\frac{d}{d s}[J(\eps(s))]=\Re\int_{\R^d}i\eps_s \left(\frac12\Delta\phi_A\bar\eps+\nabla\phi_A\nabla \bar\eps\right)dy.
\]
We replace $i\eps_s$ using  \eqref{eq:eps}. First, from standard computations 
\begin{gather*}
\Re\int_{\R^d}-\Delta\eps\left(\frac12\Delta\phi_A\bar\eps+\nabla\phi_A\nabla \bar\eps\right)dy=\int_{\R^d}\phantom{}\nabla\eps^T   \nabla^2\phi_A\nabla\bar\eps dy-\frac14\int_{\R^d}|\eps|^2\Delta^2\phi_Ady,
\\
\Re\int_{\R^d}\eps\left(\frac12\Delta\phi_A\bar\eps+\nabla\phi_A\nabla \bar\eps\right)dy=0,
\\
\frac{\lambda_s}{\lambda}\Re\int_{\R^d}i\Lambda\eps\left(\frac12\Delta\phi_A\bar\eps+\nabla\phi_A\nabla \bar\eps\right)dy=0.
\end{gather*}
Next,\[
\lambda^{\alpha}\Re\int_{\R^d}(g(P_b+\eps)-g(P_b)) \left(\frac12\Delta\phi_A\bar\eps+\nabla\phi_A\nabla \bar\eps\right)dy=O(\lambda^{\alpha}\norm{\eps}_{H^1}^2)=O(s^{-2}\norm{\eps}_{H^1}^2).
\]
The term corresponding to the second line of \eqref{eq:eps} is estimated as follows.
\begin{multline*}
\left|\dual{-i(b+\frac{\lambda_s}{\lambda})\Lambda (P_b+\eps)
+(1-\gamma_s)(P_b+\eps)
-(b_s+b^2-\theta)\frac{|y|^2}{4}P_b}
{\frac12\Delta\phi_A \eps+\nabla\phi_A\nabla \eps}\right|\\
\lesssim |\Mod(s)| \norm{\eps}_{H^1}\lesssim O(s^{-(2K+2)}).
\end{multline*}
Finally, by  \eqref{eq:error-term-estimate} and Lemma~\ref{lem:estimate-modulation-equations},
 $$
 \left| \dual{\Psi e^{-i \frac{b |y|^2}{4}}}{\frac12\Delta\phi_A\bar\eps+\nabla\phi_A\nabla \bar\eps}\right| 
 \leq O(s^{-(K+4)}\norm{\eps}_{H^1})
\leq O(s^{-(2K+4)}).$$
The result follows.
\end{proof}

\subsection{The Lyapunov property}  \label{subsec:lyapunov}

\begin{proof}[Proof of Proposition \ref{prop:Lyapunov}]
By   definition of $S$, we have
\[
\frac{d}{d s}[S(s,\eps(s))]=\frac{1}{\lambda^{4}}\left( -4\frac{\lambda_s}{\lambda} \left( H(s,\eps)+bJ(\eps) \right) + 
\frac{d}{d s} [H(s,\eps(s))]
+b\frac{d}{d s}[J(\eps(s))]+ b_sJ(\eps) \right)
\]

First, we claim the following estimate
\begin{equation}\label{eq:nnew}
\frac{d}{d s} [H(s,\eps(s))]
+b\frac{d}{d s}[J(\eps(s))]=
b\int_{\R^d}\phantom{}\nabla\eps^T   \nabla^2\phi_A\nabla\bar\eps dy-b\norm{\nabla\eps}_2^2+
\frac{b}{A}O(\norm{\eps}_{H^1}^2)+  O(s^{-(2K+3)}) .
\end{equation}
Proof of \eqref{eq:nnew}.
It is essential to see from  Lemmas \ref{lem:derivative-H} and \ref{lem:derivative-J}
that the main nonlinear terms are cancelling. Indeed, by integration by parts,
\begin{multline*}
-\Re\int_{\R^d} (f(P_b+\eps)-f(P_b)) \Lambda \bar\eps dy\\
=
-\frac d2\Re\int_{\R^d} (f(P_b+\eps)-f(P_b)) \bar\eps dy
-\Re\int_{\R^d} y\nabla(F(P_b+\eps)-F(P_b)-dF(P_b)\eps)
 dy\\+
\Re\int_{\R^d} (f(P_b+\eps)-f(P_b)-df(P_b)\eps)y\nabla\bar P_bdy
\\
=
-\frac d2\Re\int_{\R^d} (f(P_b+\eps)-f(P_b)) \bar\eps dy
+d\,\Re\int_{\R^d} (F(P_b+\eps)-F(P_b)-dF(P_b)\eps)
 dy\\+
\Re\int_{\R^d} (f(P_b+\eps)-f(P_b)-df(P_b)\eps)y\nabla\bar P_bdy,
\end{multline*}
 \begin{multline*}
-\Re\int_{\R^d} (f(P_b+\eps)-f(P_b)) \left(\frac12\Delta\phi_A\bar\eps+\nabla\phi_A\nabla \bar\eps\right)dy\\=
-\frac12\Re\int_{\R^d} (f(P_b+\eps)-f(P_b)) \Delta\phi_A\bar\eps dy
+\Re\int_{\R^d} \Delta\phi_A(F(P_b+\eps)-F(P_b)-dF(P_b)\eps)
 dy\\+
\Re\int_{\R^d} (f(P_b+\eps)-f(P_b)-df(P_b)\eps)
\nabla\phi_A\nabla \bar P_b dy.
\end{multline*}
Writing these two terms as above, it becomes clear that when $y$ or $\nabla \phi_A$ appear, they are multiplied by $\nabla P_b$, which is exponentially decaying in space (see Proposition \ref{prop:profile}). Therefore, such terms are controlled by expressions involving only $\|\eps\|_{H^1}$.\medskip

Therefore, combining Lemma \ref{lem:derivative-H} and Lemma \ref{lem:derivative-J}, we have 
\begin{multline*}
\frac{d}{d s} [H(s,\eps(s))]
+b\frac{d}{d s}[J(\eps(s))] =b\int_{\R^d}\phantom{}\nabla\eps^T   \nabla^2\phi_A\nabla\bar\eps dy
-
b\norm{\nabla\eps}_2^2
\\
+
\left(b+\frac{\lambda_s}{\lambda}\right) \Bigg(\norm{\nabla\eps}_2^2
-\frac d2\Re\int_{\R^d} (f(P_b+\eps)-f(P_b)) \bar\eps dy
+d\, \Re\int_{\R^d} (F(P_b+\eps)-F(P_b)-dF(P_b)\eps) dy
\\
+\Re\int_{\R^d} (f(P_b+\eps)-f(P_b)-df(P_b)\eps)y\nabla \bar P_b dy\Bigg)
\\
+b\Bigg(-\frac12\Re\int_{\R^d} (f(P_b+\eps)-f(P_b)) (\Delta \phi_A-d)\bar\eps dy+
\Re\int_{\R^d} (F(P_b+\eps)-F(P_b)-dF(P_b)\eps)(\Delta\phi_A-d) dy\\
+\Re\int_{\R^d} (f(P_b+\eps)-f(P_b)-df(P_b)\eps)(\nabla\phi_A-y)\nabla \bar P_b dy\Bigg)
\\
-b\frac14\int_{\R^d}|\eps|^2\Delta^2\phi_Ady +O(s^{-(2K+3)})+O(s^{-2}\norm{\eps}_{H^1}^2).
\end{multline*}
By $\left|b+\frac{\lambda_s}{\lambda}\right|\lesssim O(s^{-4})$, we have
\begin{align*}
\left|\left(b+\frac{\lambda_s}{\lambda}\right) \Bigg(\norm{\nabla\eps}_2^2
-\frac d2\Re\int_{\R^d} (f(P_b+\eps)-f(P_b)) \bar\eps dy
+d\, \Re\int_{\R^d} (F(P_b+\eps)-F(P_b)-dF(P_b)\eps) dy\right.
\\ \left.
+\Re\int_{\R^d} (f(P_b+\eps)-f(P_b)-df(P_b)\eps)y\nabla \bar P_b dy\Bigg)\right|\lesssim
s^{-4} \|\eps\|_{H^1}^{2}.
\end{align*}
Next,
\begin{multline*}
|b|  \left|-\frac12\Re\int_{\R^d} (f(P_b+\eps)-f(P_b)) \Delta(\phi_A-d)\bar\eps dy\right|\\
\lesssim \frac 1s\int_{\R^d}\left|| P|^{\frac{4}{d}}|\eps|^2|\Delta\phi_A-d|+|\eps|^{2+\frac{4}{d}} \right|dy
\lesssim \frac{e^{-\frac A2}}s \norm{\eps}_2^2+O\left(s^{-1}\norm{\eps}_{H^1}^{2+\frac 4d}\right),
\end{multline*}
and similarly for $b\Re\int_{\R^d} (F(P_b+\eps)-F(P_b)-dF(P_b)\eps)(\Delta\phi_A-d) dy $ 
and $b \Re\int_{\R^d} (f(P_b+\eps)-f(P_b)-df(P_b)\eps)(\nabla\phi_A-y)\nabla \bar P_b dy$.
Next,
\[
\left|-b \int_{\R^d}|\eps|^2\Delta^2\phi_Ady\right|\lesssim \frac{b}{A^2}\norm{\eps}_2^2.
\]
In conclusion for this term, we have obtained \eqref{eq:nnew}

Using $-\frac{\lambda_s}{\lambda} = b + O(s^{-2})$
 and  the expression of $H$ we have 
\begin{align*}
&-4\frac{\lambda_s}{\lambda} H(s,\eps) + \frac{d}{d s} [H(s,\eps(s))]
+b\frac{d}{d s}[J(\eps(s))]\\
&\gtrsim
4 b H(s,\eps) +b\int_{\R^d}\phantom{}\nabla\eps^T   \nabla^2\phi_A\nabla\bar\eps dy-b\norm{\nabla\eps}_2^2+O(s^{-2}\|\eps\|_{H^1}^2)+
\frac{b}{A}O(\norm{\eps}_{H^1}^2)+  O(s^{-(2K+3)})
\\&\gtrsim
 b\left( \int_{\R^d}\phantom{}\nabla\eps^T   \nabla^2\phi_A\nabla\bar\eps dy+\norm{\eps}_2^2-2\int_{\R^d} (F(P_b+\eps)-F(P_b)-dF(P_b)\eps)dx\right)
\\
&\qquad +2 bH(s,\eps)+
\frac{b}{A}O(\norm{\eps}_{H^1}^2)+O(s^{-(2K+3)})\end{align*}
Thus, and the coercivity properties Lemma \ref{lem:coercivity-of-H} and Lemma \ref{le:9bis}, we obtain (for $A$ large enough)
\begin{align*}
-4\frac{\lambda_s}{\lambda} H(s,\eps) + \frac{d}{d s} [H(s,\eps(s))]
+b\frac{d}{d s}[J(\eps(s))]
\gtrsim b \|\eps\|_{H^1}^2 +O(s^{-(2K+3)}).
\end{align*}
Since $b=O(s^{-1})$, $b_s=O(s^{-2})$ and $J(\eps)=O(\norm{\eps}_{H^1}^2)$, we have
$$\left( \left|\frac {\lambda_s}{\lambda}\right| b +|b_s|\right) |J(\eps)|\lesssim 
s^{-2} O(\|\eps\|_{H^1}^2)$$
and thus 
\[
\frac{d}{d s}[S(s,\eps(s))]\gtrsim
\frac{b}{\lambda^{4}}\left(\norm{\eps}_{H^1}^2+O(s^{-(2K+2)})\right).
\]
This finishes the proof.
\end{proof}

\section{End of the proof of Proposition \ref{prop:uniform-estimates-rescaled-time}}

In this section, we finish the proof of Proposition \ref{prop:uniform-estimates-rescaled-time}.
Recall from \S \ref{sec:4:3} that our objective is to prove $s_*=s_0$ by improving estimates \eqref{eq:smallness-4-s} into \eqref{eq:smallness-1-s}. Therefore, it is sufficient to prove the following lemma which closes the bounds \eqref{eq:smallness-4-s} provided $\delta(\alpha)>0$ has been chosen small enough (e.g. as in \eqref{deltaalpha}).

 \begin{lemma}[Refined estimates]\label{lem:bootstrap}  
For all $s\in[s_*,s_1]$,  
\begin{align}
& \norm{\eps(s)}_{H^1}\lesssim s^{-(K+1)},\label{eq:4s}\\
&\left|\frac{\lambda^{\frac\alpha 2}(s)}{\lambda_{\rm app}^{\frac\alpha 2}(s)}-1\right|+
\left|\frac{b(s)}{b_{\rm app}(s)}-1\right|  
\lesssim s^{-\frac 12}+s^{2-\frac 4\alpha}.\label{eq:6s}
\end{align}
\end{lemma}

\begin{proof} First, we prove \eqref{eq:4s}.
From  Proposition \ref{prop:coercivity-of-S},  and the expression of $S$, there exists a universal constant $\kappa >1$ such that for any $s\in[s_*,s_1]$,
\begin{equation}\label{eq:bootstrap-1}
\frac{1}{\kappa}\frac 1{\lambda^{4}}\left(\norm{\eps}_{H^1}^2-\kappa^2 s^{-2(K+1)}\right)\leq  S(s,\eps) \leq \frac\kappa{\lambda^{4}} \norm{\eps}_{H^1}^2.
\end{equation}
From Proposition \ref{prop:Lyapunov}, possibly taking a larger $\kappa$,
\begin{equation}\label{eq:bootstrap-2}
\frac{d}{d s}[S(s,\eps(s))]\geq \frac{1}{\kappa} \frac{b}{\lambda^{4}}\left(\norm{\eps}_{H^1}^2-\kappa^2s^{-2(K+1)}\right).
\end{equation}
Define 
\[
s_\dagger:=\inf\{ s\in[s_*,s_1],\quad \norm{\eps(\tau)}_{H^1}\leq   2\kappa^2\tau^{-(K+1)}\quad\text{for all }\tau\in[s,s_1] \}.
\]
Since $\eps(s_1)=0$, by continuity $s_\dagger$ is well-defined and  $s_\dagger<s_1$. For the sake of contradiction, assume  that $s_\dagger>s_*$. In particular, $\norm{\eps(s_\dagger)}_{H^1}=  2 \kappa^2 s_\dagger^{-(K+1)}$. Define 
\[
s_\ddagger:=\sup\{ s\in[s_\dagger,s_1],\quad \norm{\eps(\tau)}_{H^1}\geq  \kappa\tau^{-(K+1)}\quad\text{for all }\tau\in[ s_\dagger,s] \}.
\]
In particular, $s_\dagger<s_\ddagger<s_1$ and $\norm{\eps(s_\ddagger)}_{H^1}=\kappa s_\ddagger^{-(K+1)}$, and from \eqref{eq:bootstrap-2}, $S$ is nondecreasing on $[s_\dagger,s_\ddagger]$.
From equations \eqref{eq:bootstrap-1}-\eqref{eq:bootstrap-2} and the estimates on $\lambda$ (see \eqref{eq:smallness-1-s}),   we obtain 
\begin{multline*}
\norm{\eps(s_\dagger)}_{H^1}^2-\kappa^2 s_\dagger^{-2(K+1)}
 \leq 
\kappa\lambda^{4}(s_\dagger)S(s_\dagger,\eps(s_\dagger))
 \leq
\kappa\lambda^{4}(s_\dagger)S(s_\ddagger,\eps(s_\ddagger))\\
\leq \kappa^2 \frac {\lambda^{4}(s_\dagger)}{\lambda^{4}(s_\ddagger)}\|\eps(s_\ddagger)\|_{H^1}^2
\leq
\kappa^4 \frac {\lambda^{4}(s_\dagger)}{\lambda^{4}(s_\ddagger)} s_\ddagger^{-2(K+1)}
\leq
2\kappa^4 \left(\frac{s_\ddagger}{s_\dagger}\right)^{\frac{8}{\alpha}}s_\ddagger^{-2(K+1)}
\leq
2 \kappa^4  s_\dagger^{-2(K+1)},
\end{multline*} 
since $K>4/\alpha$.
Therefore $\norm{\eps(s_\dagger)}_{H^1}^2\leq 3 \kappa^4 s_\dagger^{-2(K+1)}$,
which is a contradiction. Hence $s_\dagger=s_*$ and \eqref{eq:4s} is proved.

\bigskip

Now, we prove \eqref{eq:6s}. The main idea is to use a conservation law on $(b,\lambda)$
which can be found from the differential system satisfied by $(b,\lambda)$, but that we rather derive from energy properties of the blow up profile. Recall that
$\lambda(s_1)=\lambda_1$ and  $b(s_1)=b_1$ are chosen in Lemma \ref{le:bu} so that
$\mathcal F(\lambda(s_1)) = s_1$ and $\mathcal{E}(b(s_1),\lambda(s_1))=\frac{8E_0}{\int|y|^2Q^2}.$
In particular, we deduce from \eqref{eener} that
$|E(P_{b_1,\lambda_1,\gamma_1})-E_0|\lesssim s_1^{-6}$.
Using \eqref{dener} and \eqref{eq:smallness-4-s}, \eqref{eq:estimate-modulation}, for all $s\in[s_*,s_1]$,
$$
\left|\frac{d}{ds}E(P_{b,\lambda,\gamma})\right|\lesssim s^{-(K+2)+\frac{4}{\alpha}}.
$$ 
In particular, by integration, we find, for all $s\in[s_*,s_1]$,
$|E(P_{b,\lambda,\gamma}(s))-E_0|\lesssim s^{-6}$ (recall $K>20/\alpha$)
and using \eqref{eener} at $s$,
$$
\left|\mathcal{E}(b(s),\lambda(s))-\frac{8E_0}{\int |y|^2Q^2}\right|\lesssim s^{-6}.
$$
We obtain from the expression \eqref{defee} of $\mathcal E$ with $C_0=\frac{8E_0}{\int |y|^2Q^2}$:
$$
\left|b^2-\frac{2\beta}{2-\alpha}\lambda^\alpha-C_0\l^2\right|\lesssim \frac{\l^\alpha}{s^2}
$$
where the error term $O(\frac{\l^\alpha}{s^2})$ comes from $\theta$ and cannot be improved. 
In this estimate, since $\lambda^2 \approx s^{-\frac 4\alpha}$ and
$\frac{\l^\alpha}{s^2}\approx s^{-4}$, whether or not $C_0 \lambda^2$ is controled by the error term depends on the value of $\alpha$. We address both cases at once in what follows.
Since $b\approx \lambda^{\frac \alpha2}$, 
\be
\label{neovneonvoe}
\left|b-\sqrt{\frac{2\beta}{2-\alpha}\lambda^\alpha+C_0\l^2}\right|\lesssim \frac{\l^{\frac\alpha2}}{s^2},
\ee and with $\left|\frac{\lambda_s}{\lambda}+b\right|\lesssim s^{-(K+1)}$  , we obtain
(see \eqref{defmF} for the definition of $\mathcal F$)
\be
\label{eqlfmamd}
\left|\frac{\lambda_s}{\l^{\frac{\alpha}{2}+1}\sqrt{\frac{2\beta}{2-\alpha}+C_0\l^{2-\alpha}}}+1\right|
= \left|\mathcal F'(s) - 1 \right|\lesssim s^{-2}.
\ee
Integrating \eqref{eqlfmamd} on $[s,s_1]$, we obtain
$$
\left|\mathcal F(\lambda(s_1))- \mathcal F(\lambda(s)) - (s_1-s)\right|\lesssim s^{-1}
$$
and thus, by the choice $\mathcal F(\lambda(s_1)) = s_1$, we obtain
$$\mathcal F(\lambda(s)) = s + O( s^{-1}).$$
Therefore, using \eqref{using} and the definition of $\lambda_{\rm app}(s)$ in \eqref{eq:dbu},
$$
\left|\frac{\lambda_{\rm app}^{\frac \alpha 2}(s)}{\lambda^{\frac \alpha2}(s)}-1\right| \lesssim  
s^{-\frac 12}+ s^{2 - \frac 4{\alpha}}.$$
 We reinject this estimate into \eqref{neovneonvoe} and use the definition of $b_{\rm app}$ to conclude: 
$$ b(s)=b_{\rm app}(s)+O(s^{-\frac 32}+s^{-\frac{4-\alpha}{\alpha}}).$$
This finishes the proof.
\end{proof}

\begin{appendix}

\section{Proof of Lemma \ref{lemmathreshodl}}
\label{appendixa}
By contradiction, assume that there exists a blow up solution $u(t)$ of \eqref{eq:nlseps} with $\epsilon=-1$ and $\|u(t)\|_{2}=\|Q\|_{2}$. Let a sequence $t_n\to T^*\in (0,+\infty]$ with $\|\nabla u(t_n)\|_{2}\to +\infty$ and consider the renormalized sequence 
$$
v_n(x)=\l(t_n)^{\frac d2}u(t_n,\l (t_n)x), \ \ \l(t_n)=\frac{\|\nabla Q\|_{2}}{\|\nabla u(t_n)\|_{2}}.
$$
Then, by conservation of mass, $$\|v_n\|_{2}=\|Q\|_{2}$$ and   conservation of energy and $\epsilon<0$,
$$
E_0=E(u_n)\geq E_{\rm crit}(u_n)=\frac{E_{\rm crit}(v_n)}{\l^2(t_n)}.
$$
Therefore, the sequence $v_n$ satisfies: 
$$
\|v_n\|_{2}=\|Q\|_{2}, \ \ \|\nabla v_n\|_{2}=\|\nabla Q\|_{2}, \ \ \limsup_{n\to +\infty} E_{\rm crit}(v_n)\leq 0.
$$ 
From standard concentration compactness argument, see \cite{MeRa05,We83}, there holds, up to a subsequence, for some  $x_n\in \R^d, \gamma_n\in \R$, 
$$
v_n(.-x_n)e^{i\gamma_n}\mathop{\to}_{n\to +\infty} Q\ \ \mbox{in} \ \ H^1(\R^d).
$$ 
In particular, 
$$ 
\|u(t_n)\|_{{p+1}}=\frac{\|v_n\|_{{p+1}}}{\l^{\frac{d(p-1)}{2(p+1)}}(t_n)}\to +\infty\ \ \mbox{as}\ \ n\to \infty,
$$ 
which contradicts the a priori bound from the energy conservation law and \eqref{gagenergy}: 
$$
E_0=E(u)\geq E_{\rm crit}(u)+\frac{1}{
p+1}\int |u|^{p+1}\geq \frac{1}{p+1}\int |u|^{p+1}.
$$

\section{Proof of Lemma \ref{smallsolitary}}
\label{appendixb}
For the sake of simplicity, we give the proof only for $d\geq 2$. The case $d=1$ would require an additionnal (standard) concentration compactness argument (see \cite{We83}). For $M<\|Q\|_{2}$, set $$A_M=\{u\in H^1_{\rm rad}(\R^d)\ \ \mbox{with}\ \ \|u\|_{2}=M\}$$  and consider the minimization problem $$I_M=\inf_{u\in A_M}E(u).$$ 
First, we claim
\be
\label{signim}
-\infty<I_M<0.
\ee
Indeed, from \fref{gagenergy} and 
\begin{equation}\label{GNbis}
\int |u|^{p+1} \leq C_{\rm GN}(p)\norm{\nabla u}_2^{\frac{d(p-1)}{2}}\norm{ u}_2^{p+1-\frac{d(p-1)}{2}},
\end{equation} with $1<p<1+\frac 4d$, we note that
 $I_M>-\infty$  and that any minimizing sequence is bounded in $H^1(\R^d)$.
 Let $u\in A_M$ and $v_{\lambda}(x)=\l^{\frac d2}u(\l x),$ then $v_\l\in A_M$ and $$E(v_\l)=\l^2\left[E_{\rm crit}(u)-\frac 1 {\l^{2-\frac{d(p-1)}{2}}}\frac{1}{p+1}\int|u|^{p+1}\right].$$ 
 In particular, for $0<\l\ll1 $ and $u\not \equiv0$, $E(v_\l)<0$ and \eqref{signim} follows.
 \medskip

Second, let $u_\l=\l^{\frac 2{p-1}}u(\l x)$, so that $$E(u_\l)=\l^{\frac{4}{p-1}+2-d}\left[\frac 12\int|\nabla u|^2-\frac1{p+1}\int |u|^{p+1}\right]-\frac{\l^{\frac{2}{p-1}(2+\frac{4}{d})-d}}{2+\frac 4d}\int |u|^{2+\frac4d}.$$ 
We observe that
\bee
\frac{d}{d\l}E(u_\l)_{|\lambda=1}&= & \left(\frac{4}{p-1}+2-d\right)\left[\frac 12\int|\nabla u|^2-\frac1{p+1}\int |u|^{p+1}\right]-\frac{\frac{2}{p-1}(2+\frac{4}{d})-d}{2+\frac 4d}\int |u|^{2+\frac4d}\\
& = &  \left(\frac{4}{p-1}+2-d\right)E(u)-\frac{\frac 4d}{2+\frac 4d}\left(\frac 2{p-1}-\frac d2\right)\int |u|^{2+\frac4d}.
\eee
Together with $\|u_\l\|_{2}=\l^{\frac{2}{p-1}-\frac{d}{2}}\|u\|_{2}$, which implies
$\frac{d}{d\l}{\|u_\l\|_{2}}_{|\lambda=1}>0$, this proves   that
\be
\label{monotonie}
I(M)\text{ is decreasing in }M.
\ee

To finish, let $(u_n)$ be a minimizing sequence. Up to a subsequence and from the standard radial compactness of Sobolev embeddings (see \cite{BL83}) $$u_n\rightharpoonup u\ \ \mbox{in}\ \ H^1(\R^d), \ \ u_n\to u\ \ \mbox{in}\ \ L^q, \ \ 2<q\leq 2+\frac 4d.$$ Hence $$ E(u)\leq I_M\quad \hbox{and} \quad \|u\|_{2}\leq M.$$ From \eqref{monotonie} and the definition of $I_M$, we deduce $\|u\|_{2}=M$ and $E(u)=I_M$. From a standard Lagrange multiplier argument, $u$ satisfies $$\Delta u+|u|^{1+\frac{4}{d}}u+|u|^{p-1}u=\omega u$$
for a constant $\omega\in \R$. The sign $\omega>0$ now follows from a standard  Pohozaev type argument.
\end{appendix}


\def\cprime{$'$}

\end{document}